\documentclass[11pt]{amsart}

\usepackage{amsmath,amsfonts,amssymb,amscd,amsthm,amsbsy}
\usepackage[latin1]{inputenc}   
\usepackage{graphicx,color}
\numberwithin{equation}{section}
\newcommand{\N}{\mathbb N}
\newcommand{\Z}{\mathbb Z}
\newcommand{\Q}{\mathbb Q}
\newcommand{\R}{\mathbb R}
\def\E{\mathbb E}
\def\P{\mathbb P}

\newtheorem{thm}{Theorem}[section]
\newtheorem{lem}[thm]{Lemma}
\newtheorem{cor}[thm]{Corollary}
\newtheorem{prop}[thm]{Proposition}
\theoremstyle{definition}

\def\smallnegint{\mathop{\int\mkern-13mu
        \raise.5ex\hbox{${\scriptscriptstyle\diagup}$}}\nolimits}
\def\ds{\displaystyle}

\def\ep{\varepsilon}

\def\ssetminus{\,\raise.4ex\hbox{$\scriptstyle\setminus$}\,}
\def \lg{\langle}
\def \rg{\rangle}
\newcommand{\be}{\begin{equation}}
\newcommand{\ee}{\end{equation}}
\def\QED{\begin{flushright}$\Box$\end{flushright}}

\begin{document}
\title{Homogenization and enhancement of the $G-$equation in random environments}
\author[Pierre Cardaliaguet and Panagiotis E. Souganidis]
{Pierre Cardaliaguet \and Panagiotis E. Souganidis}

\address{Ceremade, Universit\'e Paris-Dauphine,
Place du Maréchal de Lattre de Tassigny, 75775 Paris cedex 16 - France}
\email{cardaliaguet@ceremade.dauphine.fr }
\address{Department of Mathematics, University of Chicago, Chicago, Illinois 60637, USA}
\email{souganidis@math.uchicago.edu}
\thanks{Souganidis was partially supported by the National Science
Foundation. Cardaliaguet was partially supported by the ANR (Agence Nationale de la Recherche) KAMFAIBLE project  (BLAN07-3-187245).
}
\dedicatory{Version: \today}

\begin{abstract}
We study the homogenization of a  $G$-equation which is advected by a divergence free stationary vector field in a general ergodic random environment. We prove that the averaged equation
is an anisotropic deterministic G-equation and we give necessary and sufficient conditions in order to have enhancement. Since the problem is not assumed to be coercive it is not possible to have uniform bounds for the solutions. In addition, as we show, the associated minimal (first passage) time function does not satisfy, in general, the uniform integrability condition which is necessary to apply the sub-additive ergodic theorem.
We overcome these obstacles by (i) establishing a new reachability (controllability) estimate for the minimal function
and (ii) constructing, for each direction and almost surely, a random sequence which has both a long time averaged limit (due to the sub-additive
ergodic theorem) and  stays (in the same sense) asymptotically close to the minimal time.

\end{abstract}

\maketitle      

\section{Introduction}

We study the homogenization limit (averaged behavior), as $\ep\to0$, of the solution to the so-called $G-$equation
\be\label{eq:uep}
\left\{\begin{array}{l}
u_t^\ep=|Du^\ep|+\lg V(\frac{x}{\ep},\omega),Du^\ep \rg  \quad {\rm in } \quad \R^N\times (0,T),\\[2mm]

u^\epsilon = u_0 \quad \text { on } \quad \R^N\times\{0\}\ ,
\end{array}\right.\ee
set in a general stationary ergodic environment. We work in the context of viscosity solutions. Here 
$u_0$ is continuous functions and $V(y,\omega)$ is a random process as it depends on $\omega$,
an element of an underlying probability space $(\Omega, \mathcal F,\P)$. Notation and precise hypotheses are given in the following section. Here we mention that the vector field  $V$ has divergence zero and is stationary. We also remark that the results presented here can be generalized, at the cost of additional technicalities, to more sophisticated $G$-equations where $|\cdot|$ is replaced by $|(a_{i,j}(x/\ep,\omega)p_ip_j|^{1/2}$, with the matrix $(a_{i,j})$ uniformly elliptic and stationary.

The $G$-equation is a level set pde widely used as a model in ``turbulent'' combustion and flame propagation. It is related to typical discrete percolation problems but we do not expand on this here.
The level sets of the solution to \eqref{eq:uep} are supposed to be flame fronts moving in the normal direction $n$ with (normal) velocity $v_n=1- \langle V(\frac{x}{\ep},\omega),n \rangle$.
The expectation is that, in a self-averaging (stationary ergodic) environment, the oscillatory fronts will converge, as $\ep\to0$, to an averaged one moving with normal velocity $v_n={\overline H}(n)$ for an (averaged) $\overline H$ determined by the problem.
A natural question, which we answer completely here,
is  whether the presence of the advection leads to an enhancement of the front velocity, i.e., whether the averaged fronts move with normal velocity strictly larger than one.

The homogenization of the $G$-equation in random environments has been an open problem for some time. The difficulty is that, since it is not assumed that $V<1$, the equation may not be coercive. As a result there no, uniform in $\ep$, apriori (integrable) estimates controlling either the oscillations of the  solution or the growth of the associated minimal (``first passage'') time function, i.e., the ``shortest time'' it takes to connect two points using the underlying dynamics. Actually we provide a concrete example of a random environment and vector field for which such bounds are not available. The lack of such bounds
puts the study of the averaging properties of the $G$ equation outside the scope of the sub-additive ergodic theorem, which is one of the
main tool for the homogenization of Hamilton-Jacobi equations in random
media. We note that a similar problem arises in the context of identifying asymptotic shapes in the percolation theory. Overcoming it requires additional assumptions except, as far as we know, only in couple cases (see \cite{Ke} for a detailed explanation and \cite{KeS}).

To overcome this major obstacle we devise a novel strategy consisting of two steps. The first is a new reachabililty (controllability) estimate for the minimal time. The second is the construction of a random sequence, which is independent of the original probability space, along which the minimal time function has an almost sure long time asymptotic limit while it stays (in the same sense) close to the minimal time in a fixed  direction.

The first main result of this paper (Theorem~\ref{theo:hom1} stated in Section~\ref{sec:preliminaries}) is the identification of a deterministic, positively homogeneous of degree one and convex effective Hamiltonian $\overline H$, satisfying $| \overline H(p)| \geq |p|+\lg \E[V],p\rg$ for all $p\in\R^N$, combined with the assertion that, as $\ep \to 0$, the solutions $u^\ep$ of \eqref{eq:uep} converge, locally uniformly in $(x,t)$ and almost surely in $\omega$, to a deterministic function $\overline u$, the unique solution of the initial-value problem
\begin{equation} \label{HJ-eff}
\left\{ \begin{aligned}
& \overline u_t  =\overline H(D \overline u)  & \mbox{in} & \ \R^N \times (0,\infty), \\
& \overline u = u_0 & \mbox{on} & \ \R^N \times \{ 0 \}.
\end{aligned} \right.
\end{equation}

The second main result (Theorem~\ref{theo:hen1} also stated in Section~\ref{sec:preliminaries}) is the enhancement property of $\overline  H$ which asserts that the only way to have no enhancement in the direction of $p\in\R^N$, i.e., to have $\overline H(p)=|p|+\lg \E[V],p\rg$, is for the vector field $V(y,\omega)$ to be orthogonal to $p$ for all $y$ and almost surely in $\omega$. In other words the front moves faster in any direction that ``feels'' the advection.

Although there has been considerable interest recently in the homogenization of Hamilton-Jacobi, ``viscous'' Hamilton-Jacobi and uniformly elliptic second order pde in stationary ergodic environments, the analysis of the averaging behavior of \eqref{eq:uep} under general conditions has been an open problem until now. The main reason is that all previous works concerning  Hamilton-Jacobi and ``viscous'' Hamilton-Jacobi equations in random environments require that the problem is coercive in $p$, which, of course, is not the case for \eqref{eq:uep} if $\|V\| \geq 1$. It is worth noting that there are very few results for the homogenization of noncoercive problems (see Alvarez and Ishii \cite{AI01},
 Bardi and Ferrone \cite{BF}, Barles \cite{GBa}, Cardaliaguet \cite{Ca} and Imbert and Monneau \cite{IM}) all dealing with the periodic problems and none with the particular structure considered here.

The homogenization of the $G$-equation in (spatio-temporal) periodic environments, under ``a small divergence''-type assumption on $V$, as well as the enhancement properties were obtained recently by Cardaliaguet, Nolen and Souganidis in \cite{cns} --- a special case was also studied independently by Xin and Yu in \cite{XY}. The need for some additional conditions on $V$, besides boundedeness and Lipschitz continuity, to have averaging was illustrated by a specific example in \cite{cns}. The lack of coercivity is dealt with in \cite{cns} using the isoperimetric inequality which yields estimates that allow
to control uniformly the oscillations of the ``right quantities'' (see Section~\ref{sec:review} for a complete explanation).

The random setting is considerably more complicated due to the lack of compactness of the probability space. Nolen and Novikov \cite{NolenNovikov} studied \eqref{eq:uep} only in $\R^2$ and under the additional hypothesis that $V$ is the gradient of a stream function satisfying some integrability condition, a fact which, in general, is not true for stationary ergodic fields. Here we prove a general result in $\R^N$ without any restrictions on $V$.  To overcome the difficulties it is necessary to introduce new ideas and arguments which we explain later in the paper.

For completeness we refer to some of the recent work for the homogenization of stationary ergodic (degenerate) elliptic pde. While the linear case was settled long ago by Papanicolaou and Varadhan \cite{PV1,PV2} and Kozlov \cite{Ko}, and general variational problems were studied by Dal Maso and Modica \cite{DD,DD1} (see also Zhikov, Kozlov, and   Ole\u{\i}nik \cite{ZKO}), it was only relatively recently that nonlinear problems were considered (in bounded environments). Results for stochastic homogenization of Hamilton-Jacobi equations were first obtained by Souganidis \cite{S91} (see also Rezakhanlou and Tarver \cite{RT}), and for viscous Hamilton-Jacobi equations by Lions and Souganidis \cite{LS,LS03} and Kosygina, Rezakhanlou, and Varadhan \cite{KRV06}. The homogenization of these equations in spatio-temporal media was studied by Kosygina and Varadhan \cite{KV} and Schwab \cite{Sch}. More recently Armstrong and Souganidis \cite{AS} considered unbounded environments satisfying general mixing assumptions. We also mention the works of Caffarelli, Souganidis and Wang \cite{CSW} on the stochastic homogenization of uniformly elliptic equations of second-order, Caffarelli and Souganidis \cite{CS} who obtained a rate of convergence for the latter in strongly mixing environments, and Schwab~\cite{Sch1} on the homogenization of nonlocal equations. A new proof of the results of \cite{KRV06,LS03,RT,S91}, which yields convergence in probability but does not rely on formulae, was found by Lions and Souganidis \cite{LS} and was extended in \cite{AS} to almost sure. These arguments do not apply, however, to the problem  here due the lack of coercivity.

The paper is organized as follows. In the next section (Section~\ref{sec:preliminaries}), we review the notation, introduce the assumptions, and state the precise results as well as an observation about the ergodic properties of the controlled flow associated with $V$. In Section~\ref{sec:review} we recall the approach taken in \cite{cns}, \cite{XY} and \cite{NolenNovikov}, identify the additional difficulties, outline the new strategy that is needed to study the problem and provide an example of an environment and vector field showing that, in general, the integrability estimates needed to employ the sub-additive ergodic theorem are not available. In Section~\ref{sec:controllability} we present a new controllability  estimate while the homogenization is proved in Section~\ref{sec:genN}. The enhancement is studied in Section~\ref{sec:enhancement}. In the Appendix we present a simpler proof for the homogenization in $\R^2$ taking advantage of the special geometry.

\section{Preliminaries, assumptions and results} \label{sec:preliminaries}

We briefly review the basic notation used in the paper, state the precise assumptions and results and conclude with an observation  about the ergodic properties of the (controlled) flow associated to $V$, which is used several times in the paper.

\subsection{Notation and conventions}

The symbols $C$ and $c$ denote positive constants, which may vary from line to line and, unless otherwise indicated, do not depend on $\omega$. We work in the $N$-dimensional Euclidean space $\R^N$ with $N \geq 1$ and we write $\R_+=(0,\infty)$ and $\overline {\R_+}=[0,\infty)$. The sets of rational numbers and positive integers are respectively $\Q$ and $\N$.
For $x,y \in \R^N$ we denote by $|x|$ and $\lg x,y\rg $ the Euclidean norm of $x$ and the inner product of $x,y$ respectively, while $|x|_{\infty}=max_i|x_i|$.
$B(x,r)$ is the closed ball in $\R^N$ centered at $x$ and of radius $r$. We simply set $B=B(0,1)$ and $B_r=B(0,r)$ for $r\geq 0$. We also use the notation $Q_r=[-\frac{r}{2},\frac{r}{2}]^N$.
If $S_1, S_2$ are Borel measurable subsets of $\R^N$, we note by $|S_1|$ and $S_1\Delta S_2$ the Lebesgue measure of $S_1$ and
the symmetric difference between $S_1$ and $S_2$ respectively. If $f:\R^N\times \R\to \R$, we write $Df$ and $f_t$ for the space and time derivatives. For any set $A$ we write ${\bf 1}_A$ for its indicator function.
${\mathcal C}(\R^N), {\mathcal C^{1,1}}(\R^N)$ and ${\mathcal C}^1_c(\R^N)$ are respectively the spaces of bounded continuous, bounded continuously differentiable Lipschitz continuous and compactly supported continuously differentiable functions on $\R^N$. The norms of  ${\mathcal C}(\R^N)$ and ${\mathcal C^{1,1}}(\R^N)$ are  $\|\cdot\|$ and $\|\cdot\|_{{\mathcal C}^{1,1}}$. When we say that a family of functions defined on a subset $U$ of $\R^k$ converges in ${\mathcal C}(U)$, we mean that the family converges locally uniformly in $U$. If $V:\R^N\to\R$, we write $V=(V_1,...,V_N)$ and $\hat V=(V_2,...,V_N)$, and, for $u:\R^N\to \R$, $Du=(u_{x_1},...,u_{x_N})=(\partial_1u,...,\partial_Nu)$ and $\hat Du=( \partial_2u,...,\partial_Nu)$.


We emphasize that, throughout this paper, we work in the context of viscosity solutions and all differential inequalities involving functions not known to be smooth are assumed to be satisfied in the viscosity sense.  When we refer to ``standard viscosity solution theory" in support of a claim, the details can always be found in \cite{cil}.

Some additional notation and terminology will be recorded in the next subsection after we introduce the probabilistic setting.

\subsection{Assumptions and results}

The random environment is described by a probability space $(\Omega, \mathcal F, \P)$. A particular ``medium" is an element $\omega \in\Omega$. The probability space is endowed with a measurable map $\tau:\R^N \times \Omega \to \Omega$. The family $(\tau_x)_{x\in \R^N}$ is supposed to consist of  $\mathcal F$-measurable, measure-preserving transformations $\tau_x:\Omega\to \Omega$ and to satisfy, for all $x,y\in \R^N$, the group property $\tau_x\circ \tau_y=\tau_{x+y}$.

We assume that
\begin{equation} \label{eq:ass1}
\text { the family $(\tau_x)_{x\in\R^N}$ is ergodic},
\end{equation}
which means that, if $D\subseteq \Omega$ is such that $\tau_z(D) = D$ for every $z\in \R^N$, then either $\P[D] = 0$ or $\P[D] = 1$.

A $\mathcal F$-measurable function $f$ on $\R^N \times \Omega$ is said to be stationary if the law of $f(y,\cdot)$ is independent of $y$. This is quantified in terms of $\tau$ by the requirement that
\begin{equation*}
f(y,\tau_z \omega) = f(y+z,\omega) \ \mbox{for every} \ y,z\in \R^N \ \text{ and \ almost surely in \ $\omega$}.
\end{equation*}
Notice that if $\phi:\Omega \to \R^k$ is a random process, then $\tilde \phi(y,\omega) = \phi(\tau_y\omega)$ is stationary. Conversely, if $f$ is a stationary function on $\R^N \times \Omega$, then $f(y,\omega) = f(0,\tau_y\omega)$.

The expectation of a random variable $g$ with respect to $\P$ is written as $\E [g]$. If $f:\R^N \times \Omega \to \R^k$ is stationary, then $\E [f(x,\cdot)]$ is independent of $x$. In this context we abuse the notation by simply writing $\E[ f]$.

If \ $f:\R^N\times\Omega\to \R$ \ is stationary and \ $f(\cdot,\omega)\in {\mathcal C}(\R^N)$ \ a.s. in $\omega$, then
\ $\|f(\cdot,\omega)\|=\|f(\cdot,\tau_x\omega)\|$   for all  $x$  and a.s. in  $\omega$. It follows from $\eqref{eq:ass1}$ that  $\|f(\cdot,\omega)\|$  is independent of $\omega$. In this context, to simplify the notation we write $\|f\|$.

We introduce some additional notation and terminology.
Whenever  possible we abbreviate the phrase almost surely in $\omega$ by a.s. in $\omega$. Most of the statements in the paper are true a.s. in $\omega$, i.e., for all $\omega$'s in some subset $\Omega'\subseteq\Omega$ of full measure, which, of course, means that $\P[\Omega']=1$. To avoid enumerating all these subsets as we move from statement to statement, we always denote them by $\Omega_0$ with the understanding that $\Omega_0$ changes from claim to claim. The $\Omega_0$ in the main results is, of course, the intersection of the $\Omega_0$'s arising in the finitely many the steps of the proofs.

Finally, for any measurable $E\subset\Omega$ and any $\omega\in\Omega$, we denote by $E(\omega)$ the subset of $\R^N$ consisting of all the points in $\R^N$ which ``drive'' $\omega$ to $E$, i.e.,
\begin{equation}\label{specialset}
E(\omega)=\{y\in\R^N: \tau_y\omega \in E \}.
\end{equation}

We assume that:

\begin{equation} \label{eq:ass2}
\begin{cases}
V:\R^N\times\Omega \to \R^N \ \text{ is stationary} \ \text{ and, a.s. in \ $omega$,}\\[2mm]
V(\cdot,\omega) \in {\mathcal C}^{1,1}(\R^N;\R^N) \ \text{ and } \ {\rm div}V (\cdot,\omega)=0.
\end{cases}
\end{equation}

Before we continue, we emphasize that \eqref{eq:ass1} and \eqref{eq:ass2} are assumed throughout the paper and, hence, we will not repeat them in each statement. We also note that, although the homogenization result and the enhancement property only require  \eqref{eq:ass1} and \eqref{eq:ass2}, for a number of technical steps we need to assume, in addition, that the vector field has mean $0$, i.e., that
\begin{equation}\label{eq:mz}
\E[ V]=0.
\end{equation}

Next we record the constant
\begin{equation}\label{eq:M}
M=\|V\| + 1,
\end{equation}
which we will be using at several places in the paper.

Our results for \eqref{eq:uep} are:

\begin{thm}[Homogenization]\label{theo:hom1}  Assume \eqref{eq:ass1} and \eqref{eq:ass2}. There exists a set of full probability $\Omega_0 \subseteq \Omega$ and a positively homogeneous of degree one,
Lipschitz continuous, convex Hamiltonian $\overline H:\R^N\to \R$ such that, for all $p\in \R^N$, $\overline H(p)\ \geq |p|+\lg \E[V],p\rg$  and, for any $u_0\in {\mathcal C}(\R^N)$, if $u^\ep \in {\mathcal C}(\R^N\times\overline{\R_+})$ and
$\overline u \in {\mathcal C}(\R^N\times\overline{\R_+})$
are the solutions  to \eqref{eq:uep} and \eqref{HJ-eff} respectively, then, as $\ep \to 0$ and for every $\omega\in \Omega_0$ ,  $u^\ep\to \overline u$ in ${\mathcal C}(\R^N\times \overline {\R_+})$.
\end{thm}

\begin{thm}[Enhancement]\label{theo:hen1} Assume \eqref{eq:ass1} and \eqref{eq:ass2}. For any $p\in \R^N$, $\bar H(p)=|p|+\lg \E[V],p\rg $ if and only if $\lg V(x,\omega), p\rg=0$ in $\R^N$  and a.s. in  $\omega$.
\end{thm}

\subsection{A technical fact}

Let ${\mathcal A}$ be the set of time measurable maps $\alpha: \R\to B$. 
For any $(x_0,t_0,\alpha)\in \R^N\times \R\times {\mathcal A}$,  the solution  $X^{x_0,t_0,\alpha,\omega}_t:\R\to \R^N$
of the controlled system
\begin{equation}\label{eq:DS}
x'(s)=\alpha(s)+V(x(s),\omega) \ \ \text{ a.e. in $s$ } \ \text{ and }  \  x(t_0)=x_0,
\end{equation}
gives rise, for any $s,t\in \R$ such that $s\leq t$,
to a  transformation  $T_{s,t}^\alpha:\Omega\to\Omega$ defined by
\be\label{def:Talpha}
T_{s,t}^\alpha\omega= \tau_{X^{0,s,\alpha,\omega}_t}\ \omega.
\ee
When $s=0$, we abbreviate the notation by setting $T_{t}^\alpha=T_{0,t}^\alpha$.
We have:

\begin{prop}\label{prop:Talpha} For any $\alpha\in {\mathcal A}$, $(T^\alpha_{s,t})_{s,t\in\R, s\leq t}$ is a measure preserving group on $\Omega$.
\end{prop}

\noindent {\bf Proof : }
First we show that $T_{s,t}^\alpha$ is a group. The stationarity of $V$ and the uniqueness of the solutions to \eqref{eq:DS} yield
\be\label{EqTau}
X^{x_0,t_0,\alpha,\tau_y\omega}_t = X^{x_0+y,t_0,\alpha,\omega}_t-y \ \ \text{ for all } \ \ y\in \R^N\;.
\ee
Then, for $s<t<u$,   we have
$$
T_{s,u}^\alpha\omega
=
\tau_{X^{0,s,\alpha,\omega}_u}\omega
=
\tau_{X^{X_t^{0,s,\alpha,\omega},t,\alpha,\omega}_u}\omega,
$$
where, by (\ref{EqTau}), $$X^{X_t^{0,s,\alpha,\omega},t,\alpha,\omega}_u=X^{0,t,\alpha,\tau_{X_t^{0,s,\alpha,\omega}}\omega}_u+X_t^{0,s,\alpha,\omega}\;.$$
Setting, to simplify the expressions, $x_t= X_t^{0,s,\alpha,\omega}$, we find
$$
T_{s,u}^\alpha\omega
=
\tau_{X^{0,t,\alpha,\tau_{x_t}\omega}_u}\tau_{x_t}\omega
=
T_{t,u}^\alpha\circ T_{s,t}^\alpha\omega\;.
$$

To show that $T^\alpha_{s,t}$ is measure preserving, we need an intermediate step. For any measurable  $E\subset \Omega$ (recall \eqref{specialset}), we claim that
\be\label{TalphaXalpha}
(T^\alpha_{s,t}E)(\omega)=\{X^{s,y,\alpha, \omega}_t\; : y\in E(\omega)\}\;.
\ee

Indeed, if $z\in (T^\alpha_{s,t}E)(\omega)$, then, by definition, $\tau_z \omega\in T^\alpha_{s,t}E$. This means that
$T^\alpha_{t,s}(\tau_z\omega)\in E$, where $T^\alpha_{t,s}$ is the inverse of $T^\alpha_{s,t}$. But
$$
T^\alpha_{t,s}(\tau_z\omega)= \tau_{X^{t,0, \alpha,\tau_z\omega}_s}\circ \tau_z\ \omega
= \tau_{X^{t,0, \alpha,\tau_z\omega}_s+z}\ \omega = \tau_{X^{t,z, \alpha,\omega}_s}\ \omega\;.
$$

Therefore
$X^{t,z, \alpha,\omega}_s \in E(\omega)$, $z\in \{X^{s,y, \alpha,\omega}_t :  y\in E(\omega)\}$ and, hence,
$$
(T^\alpha_{s,t}E)(\omega)\subseteq \{X^{s,y,\alpha, \omega}_t : y\in E(\omega)\}\;.
$$

The other inclusion follows similarly.

Next we recall that, since $V$ is divergence free, the map $y\to X^{s,y,\alpha,\omega}_t$ preserves the Lebesgue measure in $\R^N$. Hence, for any $R>0$,
$$
\left| \{X^{s,y,\alpha, \omega}_t : y\in E(\omega)\cap B(0,R)\}\right| = \left| E(\omega)\cap B(0,R)\right|\;.
$$
Moreover, for some $C=C(s,t)>0$,
$$
\begin{array}{l}
\left|\  \{X^{s,y,\alpha, \omega}_t : y\in E(\omega)\cap B(0,R)\}\ \Delta\  (\{X^{s,y,\alpha, \omega}_t : y\in E(\omega)\} \cap B(0,R)\ )\  \right|  \\[2mm]
\qquad  \qquad \leq C |B(0,R+M(t-s))\backslash B(0,R-M(t-s))| \leq  C (R^{N-1}+1).
\end{array} $$

Using the ergodic theorem and (\ref{TalphaXalpha}), we obtain, a.s. in $\omega$,
$$
\P[E]= \lim_{R\to+\infty} \frac{\left| E(\omega)\cap B(0,R)\right|}{|B(0,R)|}
= \lim_{R\to+\infty}
\frac{|\{X^{s,y,\alpha, \omega}_t\; :\; y\in E(\omega)\} \cap B(0,R)|}{|B(0,R)|} =
\P\left[ T^\alpha_{s,t}E\right]\;.
$$
\QED
\section{review of previous results, the new strategy and an example} \label{sec:review}

We discuss here previous results, identify the difficulties and outline our strategy. Then we provide an explicit example that illustrates that the integrability property needed to apply the sub-additive is, in general, not available.

\subsection{ Review of previous results and the new strategy}

As mentioned in the introduction, the homogenization result in this paper gives a complete answer to the program initiated
\cite{cns} 
where the vector field $V$ is assumed to be periodic in space and time. A simpler space periodic case was studied in \cite{XY} while \cite{NolenNovikov} extended the result to the random setting but for $N=2$ and under an additional assumption on $V$.

The main issue concerning the asymptotics of the $G$-equation is that it is not assumed that  $\|V\| < 1$. Hence the Hamiltonian $H(x,p,\omega) = |p| + \langle V(x,\omega), p\rangle$  is not coercive in $p$ for some $x$ and $\omega$. This lack of coercivity is the main mathematical challenge in the analysis.  If either $\|V\|<1$ or the nonlinearity  has superlinear instead of linear growth (e.g., $| p|^\alpha$ with $\alpha>1$ instead of $|p|$), then $H$ is coercive in $|p|$ and the problem is within the scope of the theory developed in \cite{LPV} in the periodic and in  \cite{LS03}, \cite{LS} and \cite{S91} in the random framework.

To explain the need for the new idea/approach  we put forward here, next we describe briefly the key step of the proof in \cite{cns}. To this end, for any $P\in \R^N$ and $\lambda>0$, let $v_\lambda=v_\lambda^P$ be the unique solution to
\be\label{eqintro:lambdavlambda}
\lambda v_\lambda = |v_\lambda +P|+\lg V, Dv_\lambda +P\rg \ \text{ in } \ \R^N.
\ee
It is well known (see \cite {LS, AS}) that the homogenization of (\ref{eq:uep}) is strongly related (actually it is equivalent) to the, a.s. in $\omega$, uniform convergence in balls of radius $1/\lambda$, as $\lambda\to0$,  of $\lambda v_\lambda$ to $\bar H(P)$.

In the periodic framework, it is shown in \cite{cns} using a novel argument based on the  isoperimetric inequality for periodic sets, that the $v_\lambda$'s have bounded oscillations  uniformly in $\lambda$.  In random media, however, as far as we know, there is no isoperimetric inequality for random sets, and this technique breaks down completely.

An alternative approach, used in \cite{XY} and  \cite{NolenNovikov}, consists of analyzing the ``minimal time'' function for the controlled system \eqref{eq:DS}. To motivate the introduction of this approach, we describe next the control formulation of (\ref{eq:uep}). For simplicity we set $\ep=1$. It turns out that the solution of (\ref{eq:uep}) is given by
$$
u(x,t,\omega)= sup_{\alpha\in {\mathcal A}} u_0\left(X_t^{x,0,\alpha,\omega}\right)\;.
$$

The minimal time $\theta(x,y,\omega)$ to reach $y\in \R^N$ starting from $x\in \R^N$ is the smallest time $t$ for which there is a control $\alpha\in {\mathcal A}$ such that the solution of \eqref{eq:DS} satisfies $X_t^{x,0,\alpha,\omega}=y$.

The homogenization result of \cite{NolenNovikov} relies on some local bounds on $\theta$ obtained employing the special structure of the two-dimensional space as well as the additional assumption that vector field $V$ has mean $0$ and is the gradient of a potential $\Psi$ satisfying some integrability conditions.
Indeed, using the stream lines, i.e., the boundary of the level-lines of $\Psi$, which are generically closed curves and play the role of controllability zones, it shown in \cite{NolenNovikov}, under the additional assumptions on $V$, that
the minimal time  $\theta(x,y,\omega)$  is finite for $x,y$ and $\omega$ and, in addition, satisfies
\be \label{NolenNovikov1}
\E[ \sup_{|x|, |y| \leq R} \theta(x,y, \cdot)] <+\infty \qquad \text{ for all \ \ $R>0$},
\ee
and
\be \label{NolenNovikov2}
\lim_{R\to+\infty} \sup_{|x|, |y| \leq R}\frac{ \theta(x,y, \omega)}{R} <+\infty \;.
\ee
In view of these  estimates, it follows, using arguments introduced in \cite{S91}, that there exists  a deterministic  time constant, i.e,  for any $v\in \R^N$, the limit
$$
\lim_{t\to+\infty} \frac{\theta(0,t v,\omega)}{t}= \bar q(v)
$$
exists almost surely and is independent of $\omega$. The main tool for this is the sub-additive ergodic theorem applied to $\theta(sv,tv,\omega)$.  \\

The example in the next subsection shows that, in higher dimensions or even in two dimensions  but without that additional conditions on $\Psi$, an estimate of the form (\ref{NolenNovikov1}) does not hold in general. \\

To overcome this difficulty, we  first establish a controllability property for the system.  We prove in Theorem \ref{reach} that, if \eqref{eq:ass1}, \eqref{eq:ass2} and \eqref{eq:mz} hold, then,  \ a.s. in $\omega$ and for any $\ep>0$ small, there exists a $T(\omega,\ep)>0$ such that
\be\label{thetabound}
\theta(x,y,\omega) \leq T(\omega,\ep)+\ep|x|+(1+\ep)|y-x| \qquad \text{ for all $x,y \in \R^N$},
\ee
a fact which, of course, yields (\ref{NolenNovikov2}).

However, in contrast with \cite{NolenNovikov}, we have no control on the expectation of the constant $T(\cdot,\ep)$. Hence,  we cannot apply directly the sub-additive ergodic theorem to $\theta(sv,tv,\omega)$, since the latter requires an integrability estimates on $\theta(0, tv, \cdot)$. This is a major obstacle that we are able to overcome in this paper using a novel approach. It is worth remarking that this difficulty is also present in the theory of first passage percolation in probability in the context of finding asymptotic shapes, which are level sets of the time constant. This necessitates additional assumptions on the dynamics with the exception of a model studied by Kesten \cite{Ke} (see also \cite{KeS}) where we also refer
for more discussion about percolation.

Since \eqref{thetabound} cannot be used to derive directly the existence of a time constant, it is essential that we develop a new  argument. Indeed we identify a new quantity which (i) is sub-additive and bounded, and, hence, by the sub-additive ergodic theorem, has a long time averaged limit, which is, however, random  and (ii) stays, in an appropriate way, near $\theta$ and, hence, the latter has the same, a priori random, long time averaged limit. On the other hand, the ergodicity yields that the largest and smallest possible long time averaged limits of the minimal time must be a.s. independent of $\omega$. The result then follows.

When $N=2$, a good choice is simply the quantity $\theta(X^{\bar a}_s, X^{\bar a}_t, \omega)$, where $\bar a\in B$ and $X^{\bar a}_t$ is the solution of (\ref{eq:DS}) with initial condition $(0,0)$. Indeed, by definition, $\theta(X^{\bar a}_s,X^{\bar a}_t, \omega)$ (recall that the notation implicity assumes that $s\leq t$) is sub-additive and, obviously, bounded by $t-s$. It is also possible to check, using that  $V$ is divergence free, that it is stationary (the meaning of this is made precise later in the paper).
Moreover, $X^{\bar a}_t$ behaves almost like $t\bar a$. Indeed it is shown that, a.s. in $\omega$,
${X^{\bar a}_t}/t$ has,  as $t\to\pm\infty$, a random limit of the form $\lambda(\omega) \bar a$ with $\P[\{\omega \in \Omega : \lambda(\omega) \geq 1 \}]>0$.
Using the reachability estimate we can then prove the existence of the time constant $\overline q(\bar a)$.

In higher dimensions, the construction is far more involved because the limit of ${X^{\bar a}_t}/t$,  as $t\to\pm\infty$, is not necessarily proportional to $\bar a$. To overcome this problem we introduce a control $\alpha$ which oscillates randomly, but independently of $\omega$, around $\bar a$. In this setting we consider $\theta(X^\alpha_s,X^\alpha_t,\omega)$. 
As before, this quantity turns out to be sub-additive, bounded (again by $t-s$) and stationary in the product probability space. Using the Kakutani ergodic theorem, a more sophisticated version of the classical ergodic theorem, we show that, as $t\to\pm\infty$ and a.s. in the product space, $X_t/t$ approaches $\bar a$. Then we find that $\theta(0,X^\alpha_t,\omega)/t$ has, as $t\to\infty$ and a.s. in $\omega$, a (random) limit, which in view of the reachability estimate, is also the a.s. in $\omega$ limit of $\theta(0,t\bar a,\omega)$. We conclude, using the ergodicity, that this limit is actually independent of $\omega$ and, hence,
the time constant $\overline q(\bar a)$.

The proof of the enhancement property for periodic environments given in \cite{cns} was based on the fact that, the uniform in $\lambda$, estimate on the oscillation of the solutions $v_\lambda$ to \eqref{eqintro:lambdavlambda} yields some kind of approximate corrector. We do not have such estimates in random environments. We can use, however, the fact that we have already established the homogenization. This yields, as previously explained, the a.s. in $\omega$ and uniform in $B_{1/\lambda}$ convergence of the $\lambda v_\lambda$'s to $\overline H(p)$. It is then possible to obtain the result, after some regularizations, based on the convexity of the problem,  and several technical arguments.\\

\subsection{An example}

We work in $\R^2$. We present an example of medium and vector field for which (\ref{NolenNovikov1}) fails to hold and, more precisely, for some $x$,
$$
\E\left[  \theta(0,x,\cdot)\right]=+\infty.
$$

Let $\{e_1,e_2\}$ be the canonical basis of $\R^2$.  We have:

\begin{prop}\label{pr:ex}  There exists a probability space $(\Omega, {\mathcal F}, \P)$ satisfying \eqref{eq:ass1} and a stationary, mean zero $V_2:\R\times \Omega\to \R$ such that $V_2(\cdot,\omega)\in{\mathcal C}^{1,1}(\R)$ a.s. in $\omega$,
$\|V_2\| \leq 2$ and, if \ $\delta(\omega)=\min\{|r|, |V_2(r,\omega)|\leq1 \}, \ \E[\delta]=\infty$. Let $\theta$ be the minimal time function corresponding the vector field $V=(0,V_2)$. Then
\be\label{Cex:theta}
\ds \text{ either } \ \E\left[ \theta(0,e_2,\cdot)\right]=+\infty\ {\rm or }\  \E\left[ \theta(0,-e_2,\cdot)\right]=+\infty.
\ee
\end{prop}

Before we enter into the proof we remark that the heuristic interpretation of $\E[\delta]=\infty$ is that  $V_2$ is strictly above $1$ or strictly below $-1$ on very large intervals. 

We turn next to the 

\begin{proof}[Proof of Proposition~\ref{pr:ex}:]  First we assume that there exists a random variable $V_2$ with the properties described in the statement of the proposition, and, hence  and we prove \eqref{Cex:theta}. It is, of course, immediate that
$V=(0,V_2)$ satisfies \eqref{eq:ass2} and \eqref{eq:mz}. 

 Fix $\omega$ such that $V_2(0,\omega)>1$. Then, by definition,  $V_2(x_1,\omega)>1$ on the interval $(-\delta(\omega),\delta(\omega))$. Moreover, for any control $\alpha\in {\mathcal A}$, the first component of the solution $X^{0,0,\alpha,\omega}_t$ of \eqref{eq:DS} remains in $(-\delta(\omega),\delta(\omega))$, at least for $t\in (0,\delta(\omega))$, because its horizontal speed is at most $1$. So
$$
\frac{d}{dt} \lg X^{0,0,\alpha,\omega}_t, e_2\rg \geq  V_2(X^{0,0,\alpha,\omega}_t)-1>0\ \ \text{on \ $(0,\delta(\omega))$}\;.
$$
In particular, the second component of $X^{0,0,\alpha,\omega}_t$ remains positive on $[0,\delta(\omega))$, so that $\theta(0, -e_2,\omega)\geq \delta(\omega)$. In the same way, if $V_2(0,\omega)<-1$, then $\theta(0,e_2,\omega)\geq \delta(\omega)$. So
$$
 \theta(0,e_2,\omega)+ \theta(0,-e_2,\omega) \geq \delta(\omega)\;.
$$
Taking the expectation in the above inequality  yields (\ref{Cex:theta}). \\

Now we turn to the construction of the probability space and  $V_2$. To this end 
consider a marked  point process $(Y_n,\sigma_n)_{n\in\N}$, where $(Y_n)$ is an increasing sequence in $\R$ and the marks $\sigma_{n}$ are in $\{-1,1\}$. The map $V_2$ is constructed so that it has a constant sign $\sigma_{n-1}$ on the space intervals $(Y_{n-1}, Y_n)$, with an absolute value above $1$ in the interior of large intervals $(Y_{n-1}, Y_n)$.

Let $\mu$ be  a probability measure  on $(0,+\infty)$ such that $m= \int_0^\infty xd\mu(x)$ is finite but $\int_0^\infty x^2d\mu(x)=+\infty$. Next we define a probability measure $\hat \P$ on $\hat \Omega= ((0,+\infty)\times \{-1,1\})^{\Z}$ such that, for any element $(T_n, \sigma_n)_{n\in \Z}$ of $hat \Omega$,

\begin{enumerate}
\item $\hat \P(\sigma_0=1)=\hat \P(\sigma_0=-1)=\frac12$,

\item given $\{\sigma_0=1\}$, we have $\sigma_n=(-1)^n$ for any $n\in \Z$ while,  given $\{\sigma_0=-1\}$, we have  $\sigma_n=(-1)^{n+1}$ for any $n\in \Z$,

\item $(T_n)$ are i.i.d. with law $\mu$ and are independent of the $(\sigma_n)$.
\end{enumerate}

It is immediate that $\hat \P$ is stationary with respect to shift $\bar \tau$ on $\hat \Omega$ which is  naturally defined by $\bar \tau((T_n, \sigma_n)_{n\in \Z})= (T_{n+1}, \sigma_{n+1})_{n\in \Z}$.
We claim that $\hat P$ is actually ergodic. Indeed let $E$ be a measurable invariant with respect to  $\bar \tau$ subset of $\hat \Omega$. Then $E=E_1\cup E_{-1}$ where $E_1= E\cap \{\sigma_0=1\}$ and $E_{-1}=E\cap \{\sigma_0=-1\}$. Since $\bar \tau(E_1)=E_{-1}$,  $E_1$ and $E_{-1}$ have the same probability and $\hat \P(E)=2\hat \P(E_1)$. Moreover note  that $E_1$ is also invariant under $\bar \tau^2$. Let $F_1$ be the projection of $E_1$ onto the first component of $\hat \Omega$. Since $(t_n)_{n\in\Z} \in F_1$ if and
only if $(t_n,(-1)^n)_{n\in\Z}\in E_1$, $F_1$ is a measurable subset of $(0,+\infty)^\Z$ and satisfies a.s.  $\bar \tau^2(F_1)=F_1$, where to keep the writing simple we abuse the  notation denoting by $\bar \tau$ also  the shift on $(0,+\infty)^\Z$. The image $\hat \P_1$ of $\hat \P$ onto the first component of $\hat \Omega$ is also the product measure $\mu^{\otimes \Z}$, for which the shift $\bar \tau^2$ is ergodic. Therefore $\hat \P_1(F_1)=0$ or $1$, so that $\hat \P(E_1)= \hat \P_1(F_1)\hat \P(\{\sigma_0=1\})=0$ or $1/2$. It follows that either  $\hat \P(E)=0$ or $1$. \\

Next we describe the law $\P$ of $(Y_n, \sigma_n)_{n\in \Z}$ in such a way that the sequence $(T_n)_{n\in\Z}$ represents the intervals between the points $(Y_n)_{n\in\Z}$.  In order to get a medium which is stationary and ergodic, the law of the arrival times $(T_n)_(n\in\Z)$ has to be (slightly) distorted  under the measure $\P$ on the set $\Omega= (\R\times \{-1,1\})^{\Z}$ as follows:
\begin{enumerate}
\item $\P(\sigma_0=1)=\P(\sigma_0=-1)=\frac12$, and, given $\{\sigma_0=1\}$, $\sigma_n=(-1)^n$ for any $n\in \Z$, while, if $\{\sigma_0=-1\}$, then  $\sigma_n=(-1)^{n+1}$ for any $n\in \Z$,

\item the pair $(Y_0,Y_1)$ is independent of $(\sigma_n)_{n\in \Z}$ and
$$
\P[ -Y_0\in dx,\; Y_1-Y_0\in dv] =\frac{1}{m} {\bf 1}_{[0,v]}(ds)\mu(dv),
$$
\item for $n\in\Z\setminus \{0,1\}$, the $Y_n$'s are defined inductively  by the relation $T_n=Y_n-Y_{n-1}$, where  the $(T_n)_{n\neq 1}$ are i.i.d. with law $\mu$ and are independent of $(Y_0,Y_1)$ and $(\sigma_n)_{n\in\Z}$.
\end{enumerate}

By construction the sequence $(Y_n)_{n\in\Z}$ is increasing with $Y_0\leq 0< Y_1$. It follows (see, for example, \cite{DVJ2, Ne}) that medium constructed above is  (strictly) stationary, in the sense that the random variable $((Y_n+s,\sigma_n))_{n\in Z}$ has the same law as $((Y_n,\sigma_n))_{n\in Z}$ for any $s\in \R$. Since $\hat \P$ is the Palm measure of the process  $((Y_n,\sigma_n))_{n\in \Z}$  and, in addition, ergodic, it follows that $\P$ is ergodic too (see statement I.6.3  of \cite{BB}).

Next we claim that
$$
\E[|Y_0|\wedge Y_1]=+\infty\;.
$$

Indeed
$$
\begin{array}{rl}
\ds \E[|Y_0|\wedge Y_1] \; = &\ds  \int_0^{+\infty} \P\left[ -Y_0>x,\; Y_1>x\right]\ dx.
\end{array}
$$

Since
$$
\begin{array}{rl}
\ds \P\left[ -Y_0>x,\; Y_1>x\right]
=  & \ds \P\left[ -Y_0>x,\; Y_1-Y_0>x-Y_0\right] \\
=  & \ds \frac{1}{m}\int_0^{+\infty}\int_0^{v} {\bf 1}_{\{s>x, \; v>x+s\} }\ ds\ \mu(dv) \\
= &  \ds\frac{1}{m} \int_{2x}^{+\infty} (v-2x)\ \mu(dv),
\end{array}
$$
it follows that 
$$
\begin{array}{rl}
\ds \E[|Y_0|\wedge Y_1]\; = & \ds \frac{1}{m} \int_0^{+\infty}   \int_{2x}^{+\infty} (v-2x)\ \mu(dv)\ dx\\
= & \ds  \frac{1}{m} \int_0^{+\infty}  \frac{ v^2}{4}\ \mu(dv)\;  =\;   +\infty.
\end{array}
$$

Finally fix a  smooth, Lipschitz continuous function $\varphi:[0,+\infty)\times \R\to [0,2]$ such that
$\varphi(a,x)=0$ if $x\leq 0$ or $x\geq a$ for any $a$, and $\varphi(a,x)=2$ if  $a\geq 1$ and $x\in [1,a-1]$. We define
$$
V_2(x_1,\omega)= \sum_{n\in \Z} \varphi( Y_n(\omega)-Y_{n-1}(\omega), x_1-Y_{n-1}(\omega))\sigma_{n-1}(\omega){\bf 1}_{[Y_{n-1}(\omega),Y_n(\omega))}(x_1)\;.
$$

By construction,  $V_2$ is stationary, $|V_2|\leq 2$ and $\E[V_2]=0$. It only remains to show that the map $\delta:\Omega\to\R$ is not integrable. Let  ${\rm dist}(0,N)$ denote the distance between $0$ and the set $N=\{Y_n: \; n\in \Z\}$.  If ${\rm dist}(0,N)\geq 2$, the definition of $\varphi$  yields that $|V_2(x_1, \omega)|=2$ if
$x_1\in [-{\rm dist}(0,N)+1, {\rm dist}(0,N)-1]$, and, therefore
\be\label{deltadist}
\delta(\omega) \geq {\rm dist}(0,N)-1\;.
\ee

It follows that $\E[\delta]=\infty$, since ${\rm dist}(0,N)= |Y_0|\wedge Y_1$  and 
$$
\E\left[ {\rm dist}(0,N)\right] = \E\left[ |Y_0|\wedge Y_1\right] = +\infty\;.
$$

\end{proof}

\section{A controllability estimate}\label{sec:controllability}

For $x\in \R^N$ and $t\geq 0$ we denote by ${\mathcal R}_t(x,\omega)$ the reachable set at time $t$, which is the
is the set of points $y\in \R^N$ for which there is a
control $\alpha\in {\mathcal A}$ such that
$X^{x,0,\alpha,\omega}_t=y$. 
Similarly we denote by
${\mathcal R}_{-t}(x,\omega)$ the reachable set  at time $t$ for the controlled
system $x'=\alpha-V(x,\omega)$. 

For future use we remark here that, in view of the properties of $V$, we have, for all $R, T>0$ with $R>MT$ (recall \eqref{eq:M}) and all $t\in [0,T]$,
\be\label{eq:reach1}
{\mathcal R}_t(x,\omega) \subseteq B(x,R).
\ee

The minimal time to reach a point $y$ from a point $x$ is given by
$$
\theta(x,y, \omega)= \min\{t\geq 0\: : y\in {\mathcal R}_t(x,\omega)\}\;,
$$
where we  set $\theta(x,y, \omega)=+\infty$ if there is no $t\geq0$ with $y\in {\mathcal R}_t(x,\omega)$.

The main result of this section is the following controllability estimate:

\begin{thm} \label{reach} Assume (\ref{eq:ass1}), (\ref{eq:ass2}) and (\ref{eq:mz}). 
Then there exist $\ep_0>0$ and $\Omega_0 \subseteq \Omega$ of full probability such that,  for all $\ep\in(0,\ep_0)$ and $\omega\in \Omega_0$, there exists a constant $T(\omega, \ep)>0$ such that \eqref{thetabound} holds.
\end{thm}

Notice that although Theorem~\ref{reach} yields that the minimal time is finite a.s. in $\omega$, it does not imply any type integral bound on $\theta$.  

Some comments about assumption  (\ref{eq:mz}) are now in order. Although the mean $0$ property of $V$ is not necessary for either of our main results (Theorems \ref{theo:hom1} and \ref{theo:hen1}), Theorem \ref{reach}, on the contrary, cannot be expected to hold without some smallness assumption on $\E[V]$. Indeed, if $|\E[V]|$ is large (for instance, if $1+\|V-\E[V]\| <  |\E[V]|$), then $\theta(0, -\E[V],\omega)=+\infty$ a.s. in $\omega$ simply because any solution $X^{0,0,\alpha,\omega}$ satisfies
$$
\begin{array}{rl}
\ds \lg X^{0,0,\alpha,\omega}_t , \E[V]\rg \; =& \ds  \int_0^t \lg V(X^{0,0,\alpha,\omega}_s, \omega)+\alpha(s) , \E[V]\rg \ ds  \\
\geq & \ds t\left( |\E[V]|^2 - |\E[V]|\left( \|V-\E[V]\|+1)\right)  \right)>0\;,
\end{array}
$$
and, hence, $ X^{0,0,\alpha,\omega}_t\neq \E[V]$ for all $t\geq 0$.

Throughout this section, and unless otherwise specified, we assume that  (\ref{eq:ass1}, \ref{eq:ass2}) and (\ref{eq:mz}) hold. \\

The proof of \eqref{thetabound}  is rather long and consists of several steps some of which we identify next.
The first (Lemma~\ref{lem:VolR}) concerns the volume and perimeter of the reachable set ${\mathcal R}_t(x,\omega)$. It turns out that, for every $x$ and a.s. in $\omega$, the volume of the reachable set grows in time with a uniform lower bound of order $t^N$ 
while its perimeter is bounded from above by $s^{N-1}$ for some $s$ in each time interval of length of order $t$. These estimates combine to give that, for each $x\in\R^N$ and a.s. in $\omega$, the reachable set ${\mathcal R}_t(x,\omega)$ eventually ``fills up the whole space'', in the sense that it intersects $E(\omega)$ for any $E\subset\Omega$ with sufficiently large probability (Lemma~\ref{ReachEom}).
The second step is that, again for every $x$ and a.s. in $\omega$, the set theoretic upper and lower limits, 
as $t\to\infty$, of ${\mathcal R}_t(x,\omega)/t$ are independent of $\omega$ compact subsets ${\mathcal E}^\pm$ of $\R^N$ (Lemma~\ref{lem:E+E-indepom}) and, in addition, ${\mathcal E}^{-}$ is convex (Lemma~\ref{lem:convex}) and contains the unit ball $B$ (Lemma~\ref{lem:B01subsetE-}). The last step is to ``transform'' the set-theoretic limits to limits in terms of measure (Lemma~\ref{lem:RtB01}).

We begin with

\begin{lem}\label{lem:VolR} There exists $\Omega_0\subseteq\Omega$ of full probability such that, for all $t\in \R\setminus\{0\}, x\in \R^N$ and $\omega\in \Omega_0$, ${\mathcal R}_t(x,\omega)$ has a non empty interior and a
boundary with a finite perimeter. Moreover, there exists $\beta>0$ depending on $N$
such that
$$
\left|{\mathcal R}_t(x,\omega)\right|\geq \beta |t|^N.
$$
Finally, for any $0<\gamma_1<\gamma_2$, there exists $K=K(\gamma_1,\gamma_2)>0$ such that, for all $t>0$,
$$
{\rm Per}({\mathcal R}_s(x,\omega)) \leq Ks^{N-1} \ \ \text{ for some \ $s\in [\gamma_1 t, \gamma_2 t]$}.
$$
\end{lem}

{\bf Proof : } We only present the proof for $t>0$ since the argument  for $t<0$ is similar.

Throughout the proof we fix $\omega$ such that $V(\cdot,\omega) \in {\mathcal C}^{1,1}(\R^N)$ and a horizon $T>0$.
It follows from the regularity of $V(\cdot,\omega)$ (see Cannarsa and Frankowska \cite{cf06}) that,
for any $0<\tau<T$, there exits a constant $r=r(\tau, T)>0$ such that, for all $t\in [\tau, T]$, ${\mathcal R}_t(x,\omega)$
has the interior ball property of radius $r$, i.e.,
for all $z\in \partial {\mathcal R}_t(x,\omega)$,
there exists $y\in \R^N$ such that $z\in B(y,r) \subset {\mathcal R}_t(x,\omega).$
In particular, for all $t>0$, the set
${\mathcal R}_t(x,\omega)$ has a non empty interior and a finite perimeter.

Arguing exactly as in the proof of Lemma 3.2 of \cite{cns} we find that, for any $\varphi\in {\mathcal C}^1_c(\R^N)$, the map
$t\to
I(t)=\int_{{\mathcal R}_t(x,\omega)}\varphi(y)dy,
$
is absolutely continuous and, for almost all $t\geq0$,
$$
\frac{d}{dt} I(t)=\int_{ \partial {\mathcal R}_t(x,\omega)}\varphi(y)(1-\lg V(y,\omega),\nu_y\rg )d{\mathcal H}^{N-1}(y),
$$
where $\nu_y$ is the measure theoretic outward unit normal of
${\mathcal R}_t(x,\omega)$ at $y$.

Choose $\varphi\in{\mathcal C}^1_c(\R^N)$ such that $\varphi=1$ in $B(x, R)$
where $R> MT$. It follows from \eqref{eq:reach1} that 
$I(t)=|{\mathcal R}_t(x,\omega)|$ and
$$
\frac{d}{dt} |{\mathcal R}_t(x,\omega)| =\int_{ \partial {\mathcal R}_t(x,\omega)}(1-\lg V(y,\omega)), \nu_y\rg)d{\mathcal H}^{N-1}(y).
$$

The facts that  $V$ is divergence free and ${\mathcal R}_t(x,\omega)$ has a finite perimeter yield
$$
\int_{ \partial {\mathcal R}_t(x,\omega)}\lg V(y,\omega)), \nu_y\rg d{\mathcal H}^{N-1}(y)=0.
$$

Let $c_I$ be the constant in the isoperimetric inequality in $\R^N$, i.e., $c_I$ is the smallest constant such that  $|E|^{(N-1)/N}\leq c_I {\mathcal H}^{N-1}(\partial E)$ for any Borel set $E\subset\R^N$. It follows that
$$
{\mathcal H}^{N-1}(\partial {\mathcal R}_t(x,\omega)) \geq \frac{1}{c_I}|{\mathcal R}_t(x,\omega)|^{(N-1)/N}.
$$

Thus
$$
\frac{d}{dt} |{\mathcal R}_t(x,\omega)| \geq \frac{1}{c_I}|{\mathcal R}_t(x,\omega)|^{(N-1)/N}\;,
$$
and, since $|{\mathcal R}_t(x,\omega)|>0$ for any $t>0$,
$$
|{\mathcal R}_t(x,\omega)|\geq \left(\frac{t}{Nc_I}\right)^N\qquad \text{ for all \  $t>0$}.
$$

A slight modification of the above argument yields the estimate for the perimeter.
Indeed, assume that, for some constant $K>0$ to be chosen later,
$$
{\rm Per}({\mathcal R}_s(x,\omega)) \geq Ks^{N-1} \ \  \text{ for all} \ \ s \in [\gamma_1 t, \gamma_2t].
$$

Since 
$$
\frac{d}{ds} |{\mathcal R}_s(x,\omega)| ={\rm Per}({\mathcal R}_s(x,\omega)),
$$
integrating over  $[\gamma_1 t, \gamma_2t]$ and using the lower bound on the perimeter we get
$$
|{\mathcal R}_{\gamma_2 t}(x,\omega)| \geq \frac{K}{N} (\gamma_2^N-\gamma_1^N) t^N\;,
$$
while
$$
\left|{\mathcal R}_{\gamma_2 t}(x,\omega)\right| \leq |B(x, M\gamma_2t)|\leq (M\gamma_2t)^N|B(0,1)|\;.
$$
This leads to a contradiction if $K$ is chosen so that
$$
K> N|B(0,1)| \ \frac{(M\gamma_2)^N}{\gamma_2^N-\gamma_1^N}\;.
$$
\QED

As a consequence we have:
\begin{lem} \label{ReachEom}
There exist constants $\delta_0\in (0,1)$ and $k>0$, and, for any $\delta\in (0,\delta_0)$ and any measurable set $E\subset \Omega$
with $\P[E]\geq 1-\delta$, a subset $\Omega'\subseteq\Omega$ of full probability such that, for all $\omega\in \Omega'$, there exists a constant $K=K(E, \omega,\delta)>0$ such that, for all $x\in\R^N$,
$$
\left| {\mathcal R}_t(x,\omega) \cap E(\omega)\right|>0\qquad {\rm for }\; t=\pm(K+k\delta^{1/N}|x|)\;,
$$
and, in particular, ${\mathcal R}_t(x,\omega) \cap E(\omega)\neq \emptyset$.
\end{lem}


\noindent {\bf Proof : } We only present the proof for $t={K+k\delta^{1/N}|x|}$. Since $\P[E]\geq 1-\delta$, the ergodic theorem yields $\Omega'\subseteq\Omega$ of full probability such that, for all   $\omega\in\Omega'$,  there exists some $R_0=R_0(\omega)>0$ such that, for all  $R\geq R_0$,
\be\label{VolECapQR}
 \left| Q_R\cap E(\omega)\right| \geq (1-2\ep)|Q_R| \;. 
\ee

Also note that, for all $t\geq0$,
$$
{\mathcal R}_t(x,\omega)\subset Q_{|x|+Mt} 
$$

Choose $R=|x|(1+2M(2\delta/\beta)^{1/N})+R_0$ and $t =(R-|x|)/M$, where $\beta$ is defined in  Lemma \ref{lem:VolR}.
Then
$$
t= R_0/M+ 2(2\delta/\beta)^{1/N}|x|\;,
$$
which is of the form $K+k\delta^{1/N}|x|$ as in the claim. It follows that, as soon as $2M(2\delta/\beta)^{1/N}< 1$, an inequality that fixes $\delta_0\in (0,1)$,
\be\label{t>R}
t > R(2\delta/\beta)^{1/N}.
\ee

Indeed
\begin{eqnarray*}
t =(R-|x|)/M > 2|x|\left(\frac{2\delta}{\beta}\right)^{1/N} = \frac{R-R_0}{  1+2M(2\delta/\beta)^{1/N}}2\left(\frac{2\delta}{\beta}\right)^{1/N}\\
 \qquad > R \left(\frac{2\delta}{\beta}\right)^{1/N} \frac{2}{1+2M(2\delta/\beta)^{1/N}} > R \left(\frac{2\delta}{\beta}\right)^{1/N}
 \end{eqnarray*}

We now claim that $\left| {\mathcal R}_t(x,\omega) \cap E(\omega)\right|>0$.
If not, since ${\mathcal R}_t(x,0)\subset Q_{|x|+Mt}=Q_R$, 
$$
R^N=|Q_R|\geq \left|{\mathcal R}_t(x,0)\cup (E(\omega)\cap Q_R)\right|=
\left|{\mathcal R}_t(x,0)\right|+ \left|E(\omega)\cap Q_R\right|\;,
$$
while, in view of  Lemma~\ref{lem:VolR} and (\ref{VolECapQR}) and (\ref{t>R}),
$$
\left|{\mathcal R}_t(x,0)\right|+ \left|E(\omega)\cap Q_R\right|
\geq
\beta t^N + (1-2\delta) |Q_R| > R^N\;,
$$
which is not possible. 
\QED

We continue with a technical lemma (Lemma \ref{SubSolOmega}) which does not require $V$ to have mean $0$ and is essential for the sequel, as it serves as a tool later in the paper to prove that some random quantities are actually independent of $\omega$. It also provides a link between the ergodic properties of the flow \eqref{eq:DS}  and time-independent super-solutions of the $G$-equation.

Before the statement and proof we recall the notions of inf-convolution and sup-convolution which are a very basic tool of the theory of viscosity solutions (see \cite{cil}). To this end, let $f:\R^N\to \R$ be bounded and lower semicontinuous (to define the inf-convolution) or upper semicontinuous (to define the sup-convolution). For each $\eta>0$ the inf-convolution $f_\eta:\R^n\to\R$ and the sup-convolution $f^\eta$ of $f$ are
\be\label{eq:conv}
f_\eta(x)=\inf_{y\in\R^N}\{f(y) + (2\eta)^{-1}|x-y|^2\} \ \ \text{ and } \ \ f^\eta(x)=\sup_{y\in\R^N}\{f(y) - (2\eta)^{-1}|x-y|^2\}.
\ee

It turns out that $f_\eta$ (resp. $f^\eta$) is a bounded, Lipschitz continuous, semiconvex (resp. semiconcave) approximation from below (resp. above), as $\eta\to0$, of $f$ which also preserves the notion of supersolution (resp. subsolution).

We also need the following technical remark concerning the properties of the flow \eqref{eq:DS}.

\begin{lem}\label{lem:cone} There exists $\eta>0$ such that, for all $x\in \R^N$ and a.s. in $\omega$, the open cone
$$
C_\eta(x)=\{ x+s(-V(x,\omega)+\eta b)\; : \; s\in (0,\eta), \; b\in B\}
$$
is contained in  $\bigcup_{t\in(0,1)} {\mathcal R}_{-t}(x,\omega)$.
\end{lem}

\noindent {\bf Proof : } For $b\in B$, $s\in [0,\eta]$ and $\eta\in (0,1)$ to be chosen below, let $X_s= x+ s(-V(x,\omega)+\eta b)$. Then
$$
|X_s'+V(X_s,\omega)| \leq \|DV\||X_s-x|+ \eta \leq ( \|DV\|(\|V\|+\eta)+1) \eta \;.
$$
We choose  $\eta$ such that the right-hand side is less than one. Then, setting $\alpha_s= X_s'+ V(X_s,\omega)$, we have $\alpha\in {\mathcal A}$  and $X_s=X_{-s}^{x,0,\alpha, \omega}$. Therefore $X_s\in {\mathcal R}_{-s}(x,\omega)$ for $s\in (0,\eta)$.
\QED

We have:

\begin{lem}\label{SubSolOmega} Assume \eqref{eq:ass1} and \eqref{eq:ass2}. (i) Let $w:\R^N\times \Omega\to \R$ be a bounded, stationary, lower semicontinuous, super-solution of
\be\label{HJineq}
0 \geq |Dw|+\lg V,Dw\rg \qquad {\rm in}\; \ \R^N\;.
\ee
Then $w$ is constant with respect to both $x$ and $\omega$, i.e, there exists a set  $\Omega_0\subseteq \Omega$ of full probability and a constant $c$ such that
$w(x,\omega)=c$ for all $(x,\omega)\in \R^N\times \Omega_0$.

\indent (ii) If $w:\R^N\times \Omega\to \R$ is bounded, stationary and, a.s. in $\omega$, for any control $\alpha\in {\mathcal A}$, any $x\in \R^N$
 the map $t\to w(X_t^{x,0,\alpha,\omega},\omega)$ is nonincreasing, then $w$
is constant, i.e.,  there exists a set $\Omega_0\subseteq \Omega$ of full probability and a constant $c$ such that
$w(x,\omega)=c$ for all $(x,\omega)\in \R^N\times \Omega_0$.
\end{lem}

\noindent {\bf Proof: }(i) To prove the claim we need to use  the ergodic theorem which requires some integration. On the other hand \eqref{HJineq} holds only in the viscosity sense. It is therefore necessary to regularize the $w$ so that we can have an inequality like  \eqref{HJineq} holding a.e. in $x$ and a.s. in $\omega$. Following an argument used in \cite{XY} for the periodic setting,  for $\theta>0$ we consider the inf-convolution $w_\theta$ of $w$ and observe that,
for any $\delta>0$ small, it is possible to choose $\theta$ sufficiently small (depending on $\delta$) so  that $w_\theta(\cdot,\omega)$ is Lipschitz continuous uniformly in $\omega$ and satisfies, a.e. in $x$ and a.s. in $\omega$,
$$
0\geq (1-\delta)| Dw_\theta| + \lg V,Dw_\theta\rg \;.
$$
Integrating over $B(0,R)$ and using that $w_\theta$ is bounded and $V$ is divergence free yields
$$
0\geq (1-\delta) \int_{B(0,R)} | Dw_\theta|-CR^{N-1}\;.
$$
Since $Dw_\theta$ is stationary and bounded, dividing the above inequality by $|B(0,R)|$ and letting $R\to +\infty$, we find, making use of  the ergodic theorem, that
$$\E[|Dw_\theta|]=0.$$ This implies that $w_\theta$ is constant, i.e.,  there exits $\Omega_\theta\subseteq \Omega$ of full probability and $c_\theta\in \R$ such that $w_\theta(x,\omega)=c_\theta$ for any $(x,\omega)\in \R^N\times \Omega_\theta$. Choosing $\theta_n\to 0$ such that $c=\lim_n c_{\theta_n}$ (recall that the $w_\theta$'s are bounded independently of $\theta$) exists and setting $\Omega_0=\bigcap_n \Omega_{\theta_n}$, we obtain that $\P[\Omega_0]=1$ and $w(x, \omega)=c$ for any $(x,\omega)\in \R^N\times \Omega_0$.

(ii) It follows 
that the lower semicontinuous envelope of $w_*$ of $w$ is also  bounded, stationary and such that, a.s.,
the map $t\to w_*(X_t^{x,0,\alpha,\omega},\omega)$ is non-increasing for all controls $\alpha\in {\mathcal A}$ and $x\in \R^N$.
This implies that $w_*$ is a lower semicontinuous viscosity super-solution of (\ref{HJineq}). 

According to the first part of the Lemma, $w_*$ is constant, i.e.,
 there exists a set of full measure $\Omega_0\subseteq \Omega$ and a constant $c$ such that \
$w_*(x,\omega)=c$ \ for all $(x,\omega)\in \R^N\times \Omega_0$.  We claim that $w$ must be constant itself, i.e., that \ $w(x,\omega)=c$ \
for all $(x,\omega)\in \R^N\times \Omega_0$. Indeed if not, there must exist some $(x,\omega)\in \R^N\times \Omega_0$ and $\ep>0$ such that $w(x,\omega)> c+\ep$. Note that the set $\{\omega\in\Omega: w(y, \omega)>c+\ep\}$ is invariant for the map $X_{-t}^{y,0, \alpha,\omega}$ for all $y$ and $\alpha\in{\mathcal A}$, because
the map $t\to w(X_{-t}^{y,0, \alpha,\omega}, \omega)$ is nondecreasing. Lemma \ref{lem:cone} states that  there exists some $\eta>0$ such that, for all  $x$ and $\omega$, the small open cone $C_\eta(x)=x+(0,\eta)(-V(x,\omega)+\eta B)$  is contained in $\bigcup_{t\in (0,1)} {\mathcal R}_{-t}(x,\omega)$. Then $w(\cdot,\omega)>c+\ep$ in $C_\eta(x)$, and, hence, $w_*(\cdot,\omega)\geq c+\ep$ in $C_\eta(x)$. This is impossible since $w_*=c$.
\QED

Having basically concluded with all the preliminaries we now turn to the essential part of this section, i.e., the long time averaged properties of the reachable set ${\mathcal R}_t(x,\omega)$. To this end,
let
$$
{\mathcal E}^+(x,\omega) \;  = \; \{ z\in \R^N : \text{there exist} \; t_n\to +\infty \; \text{ and } \;  z_n\in \frac{1}{t_n}{\mathcal R}_{t_n}(x,\omega)\; \text{ such  that } \  z_n\to z\}
$$
and
$$
{\mathcal E}^-(x,\omega)\;  = \; \{ z\in \R^N : \text{ there exist } \  z_t\in \frac{1}{t}{\mathcal R}_t(x,\omega)\; \text{ such  that  } \  z_t\to z\; {\rm as}
\; t\to +\infty\}
$$
be the upper and lower Kuratowski limits, as $t\to\infty$,  of the sets ${\mathcal R}_t(x,\omega)/t$. It follows easily that
${\mathcal E}^\pm(x,\omega)$ are compact subsets of $\R^N$ and, in addition, that
${\mathcal E}^+(x,\omega)$ is nonempty. Our aim is to show that (i) the sets ${\mathcal E}^\pm(x,\omega)$
are in fact independent of $x$ and $\omega$, and (ii) ${\mathcal E}^-(x,\omega)$ is convex and contains the closed unit ball $B$ of $\R^N$. The proof of the equality  between ${\mathcal E}^+$ and ${\mathcal E}^-$ is established in the next section (Theorem \ref{thm:TimeConstant}). \\

We have:

\begin{lem}\label{lem:E+E-indepom} There exist compact sets ${\mathcal E}^\pm \subset \R^N$ and a set $\Omega_0\subseteq \Omega$ of full probability  such that, for all $(x,\omega)\in \R^N\times \Omega_0$,
$
{\mathcal E}^\pm(x,\omega) = {\mathcal E}^\pm .
$
\end{lem}

\noindent {\bf Proof: } Since the arguments are similar, here we prove the claim only for ${\mathcal E}^+$.
To this end, for any nonnegative $\varphi\in {\mathcal C}(\R^N)$ we define the measurable map $w_\varphi :\R^N\times \Omega\to\R$ by
$$
w_\varphi(x,\omega)= \max_{z\in {\mathcal E}^+(x,\omega)}\varphi(z)\;.
$$

Since, in view of the stationarity of $V$
$$
{\mathcal R}_t(x+y,\omega)={\mathcal R}_t(x,\tau_y\omega) +y\;,
$$
we have
$$
{\mathcal E}^+(x+y,\omega)= {\mathcal E}^+(x,\tau_y\omega)\;,
$$
and, hence, $w_\varphi$ is stationary. Moreover, noting that, for any $0\leq s\leq s'$,
$$
{\mathcal R}_t(X^{x,0,\alpha,\omega}_{s'},\omega) \subset {\mathcal R}_{s'-s+t}(X^{x,0,\alpha,\omega}_{s},\omega),
$$
we also get
$$
{\mathcal E}^+(X^{x,0,\alpha,\omega}_{s'},\omega)\subset {\mathcal E}^+(X^{x,0,\alpha,\omega}_{s},\omega)\;.
$$
Therefore the map $s\to w_\varphi(X^{x,0,\alpha,\omega}_{s}, \omega)$ is nonincreasing for any $x\in\R^N$ and $\alpha \in {\mathcal A}$. Lemma \ref{SubSolOmega} then yields the existence of a constant $c_\varphi$ and a set $\Omega_\varphi\subseteq \Omega$ of full probability such that, for all $(x,\omega)\in \R^N\times \Omega_\varphi$, $w_\varphi(x,\omega)= c_\varphi$.
Let $\varphi_n$ be dense in ${\mathcal C}(B_{M+1}, \R^+)$ and set $\Omega_0=\bigcap_n \Omega_{\varphi_n}$. Then $\P[\Omega_0]=1$ and, since $\varphi\to w_\varphi(x,\omega)$ is uniformly continuous with respect to $(x,\omega)$,  for any $\varphi$ there exists a constant $c_\varphi$ such that $w_\varphi(x,\omega)=c_\varphi$ for all $(x,\omega)\in \R^N\times \Omega_0$.

Finally, we define
$$
{\mathcal E}^+=\left\{ z\in\R^N\ : c_\varphi\geq \varphi(z) \ \  \text{ for all } \; \varphi\in {\mathcal C}(B_{M+1};\R_+)\right\},
$$
and claim that ${\mathcal E}^+= {\mathcal E}^+(x,\omega)$ for all $(x,\omega)\in \R^N\times \Omega_0$. Indeed fix $(x,\omega)\in \R^N\times \Omega_0$. If $z\notin {\mathcal E}^+$, there must exist $\varphi\in  {\mathcal C}(B_{M+1};\R_+)$ such that
$c_\varphi< \varphi(z)$. Since $ c_\varphi= w_\varphi(x,\omega)= \max_{z'\in {\mathcal E}^+(x,\omega)}\varphi(z')$, this implies that $z\notin {\mathcal E}^+(x,\omega)$. If $z\in {\mathcal E}^+$, let  $\varphi\in   {\mathcal C}(B_{M+1};[0,1])$ be such that $\{\varphi=1\}= \{z\}$. Then $1=\varphi(z)\leq c_\varphi= \max_{z'\in {\mathcal E}^+(x,\omega)}\varphi(z')\leq 1 $, and, hence, $z\in {\mathcal E}^+(x,\omega)$.
\QED

We remark that, although  the compact sets ${\mathcal E}^\pm$ are independent of $\omega$, they are defined for all $\omega$'s in some
$\Omega_0\subseteq\Omega$ with $\P[\Omega_0]=1$. In view of this, the reader should keep in mind that, even though there is no reference to $\omega$ in the statements of the next two Lemmata which apply to ${\mathcal E}^-$, the claims are valid for those $\omega$'s for which ${\mathcal E}^-$ is identified.

\begin{lem}\label{lem:convex}
The set ${\mathcal E}^-$ is convex.
\end{lem}

\noindent {\bf Proof: } Let $\Omega_0$ be the set of full probability where ${\mathcal E}^-$ is defined and $\delta_0>0$ and $k$ be the two fixed constants in Lemma~\ref{ReachEom}.  For  $\ep\in(0,\delta_0)$  choose $T_\ep$ so large that $\P[S_\ep]>1-\ep$ where
$$
S_\ep=\left\{\omega\in \Omega_0 : {\mathcal E}^-\subset \frac{1}{t}{\mathcal R}_t(0,\omega)+\ep B \ \  \text{ for all} \ \ t\geq T_\ep\right\},
$$
and set $S_\ep(\omega)=\{x\in \R^N\; : \; \tau_x\omega\in S_\ep\}$.

Since $\P[S_\ep]>1-\ep$, it follows from Lemma \ref{ReachEom} that there exists a subset $\Omega'_\ep=\Omega'(S_\ep)$ of $\Omega_0$ of full probability such that, for all $\omega\in\Omega_\ep'$,
there exists a constant $K=K(S_\ep,\omega,\ep)>0$ such that, for all $x\in \R^N$,
$$
{\mathcal R}_{K+k\ep^{1/N}|x|}(x,\omega) \cap S_\ep(\omega)\neq \emptyset\;.
$$

Next we fix $\omega\in S_\ep\cap\Omega'_\ep$ and $y_1, y_2\in {\mathcal E}^-$. We  show that, for some $C>0$, $y_3=(y_1+y_2)/2 \in {\mathcal E}^-+C\ep^{1/N} B$. This yields the convexity of ${\mathcal E}^-$, since
${\mathcal E}^-$  is compact and $\ep$ is arbitrary.

Let $\bar t= t/(2+K+Ck\ep^{1/N})$ for $t$ large enough so
that $\bar t\geq T_\ep$. We construct a control $\alpha:[0,t] \to B$ by gluing together three other ones chosen as follows. By the definition of $S_\ep$, there
exits
$\alpha_1 \in {\mathcal A}$ such that
$$
\left|\frac{1}{\bar t}X_{\bar t}^{0,0,\alpha_1,\omega}-y_1\right|\leq \ep\;.
$$
Note that, since $x_1= X_{\bar t}^{0,0,\alpha_1,\omega}$  satisfies $|x_1|\leq M\bar t$ ,
there exits $\alpha_2 \in {\mathcal A}$ and a time $t_2=K+Ck\ep^{1/N}\bar t$ such that
$x_2=X_{t_2}^{x_1,0, \alpha_2, \omega}\in S_\ep(\omega)$. Moreover, since
$
\omega_2=\tau_{x_2}\omega\in S_\ep,
$
there exits some $\alpha_3 \in {\mathcal A}$ such that
$$
\left|\frac{1}{\bar t}X_{\bar t}^{0,0,\alpha_3, \omega_2}-y_2\right|\leq \ep\;,
$$
and, in addition,  $X_{\bar t}^{0,0,\alpha_3, \omega_2}=X_{\bar t}^{x_2,0,\alpha_3, \omega}-x_2$.

We set
$$
\alpha(s)=\left\{\begin{array}{cl}
\alpha_1(s) & {\rm on } \ \ [0,\bar t),\\
\alpha_2(s-\bar t) & {\rm on } \ \  [\bar t, \bar t+t_2),\\
\alpha_3(s-(\bar t+t_2)) & {\rm on} \ \ [\bar t+t_2,+\infty).
\end{array}\right.
$$

In view of the definition of $\bar t$,  we have $t= 2\bar t+t_2$ and
$
X_{2\bar t+t_2}^{0,0,\alpha,\omega}=X_{\bar t}^{0,0,\alpha_3,\omega_2}+X_{t_2}^{x_1,0,\alpha_2,\omega},
$
while
$$
\left|X_{t_2}^{x_1,0,\alpha_2,\omega}-x_1\right|\leq C t_2 \;.
$$
Hence
$$
\begin{array}{l}
\ds{ \left|\frac{1}{t} X_{t}^{0,0,\alpha,\omega}-\frac{y_1+y_2}{2}\right|}\\
\qquad \leq
\ds{
\left|\frac{1}{2\bar t+t_2}-\frac{1}{2\bar t}\right||X_{2\bar t+t_2}^{0,0,\alpha,\omega}|
+
\left|\frac{1}{2\bar t} X_{\bar t}^{0,0,\alpha_3,\omega_2}-\frac{y_2}{2}\right|
+
\left|\frac{1}{2\bar t} X_{t_2}^{x_1,0,\alpha_2,\omega}-\frac{y_1}{2}\right| } \\
\qquad \leq \ds{ \frac{t_2}{2\bar t(2\bar t+t_2)} M(2\bar t+t_2)
+ \frac{\ep}{2}+ \frac{Mt_2}{2\bar t}+ \frac{\ep}{2}\; \leq \; C\ep^{1/N}
},
\end{array}
$$
with the last inequality holding for $\bar t$ (or $t$) large enough since $t_2\leq  C(K+k\ep^{1/N}\bar t)$. \QED

We continue with 

\begin{lem}\label{lem:B01subsetE-} The closed unit ball $B\subset\R^N$ is contained in ${\mathcal E}^-$.
\end{lem}

\noindent {\bf Proof: } Let $\alpha\in B$ be a constant control. Then the transformation $T_t^\alpha:\Omega\to\Omega$ defined by
(\ref{def:Talpha}) is autonomous and, according to Proposition \ref{prop:Talpha}, is measure preserving. Let  ${\mathcal F}^\alpha$ be the $\sigma-$algebra of the invariant sets for
$T_t^\alpha$.  Using the ergodic theorem we find that, as $t\to\infty$ and a.s. in $\omega$,
\be\label{ergo+}
\frac{1}{t} X^{0,0,\alpha,\omega}_t= \frac{1}{t}\int_0^t (V(X^{0,0,\alpha,\omega}_s,\omega)+\alpha)ds=
\frac{1}{t}\int_0^t (V(T_s^\alpha\omega)+\alpha)ds
\to \E\left[ V+\alpha\; |\; {\mathcal F}^\alpha\right]. 
\ee
We also note  for later use that, since the invariant sets for $V+\alpha$ and for $-(V+\alpha)$ are the same, a.s. in $\omega$,
\be\label{ergo-}
\lim_{t\to-\infty} \frac{1}{t} X^{0,0,\alpha,\omega}_t=
\lim_{t\to+\infty} -\frac{1}{t} X^{0,0,\alpha,\omega}_{-t}= - \E\left[ -(V+\alpha)\; |\; {\mathcal F}^\alpha\right]=
\lim_{t\to+\infty} \frac{1}{t} X^{0,0,\alpha,\omega}_t. 
\ee

It follows from  (\ref{ergo+}) that, a.s. in $\omega$,
$
\E\left[ V+\alpha\; |\; {\mathcal F}^\alpha\right] \in {\mathcal E}^- .
$
Since ${\mathcal E}^-$ is convex and compact and $V$ has mean zero, the last claim implies that
$
\alpha= \E\left[\E\left[ V+\alpha\; |\; {\mathcal F}^\alpha\right]\right] \in {\mathcal E}^-\;.
$\QED

The final step is to transform the set theoretic Kuratowski limits given by Lemma \ref{lem:B01subsetE-} into a limit in terms of measure.

We have:

\begin{lem}\label{lem:RtB01}
There exists $\Omega_0\subseteq\Omega$ of full probability such that, for every $\omega\in\Omega_0$  there exist a positive constant $k$ and,  for any $\delta\in (0,1/k^N)$,
a (measurable with respect to $\omega$) $T=T(\omega,\delta) >0$ such that, for $t\geq T(\omega,\delta)$, there exists $t_1\in ((1-k\delta^{1/N})t,t)$ such that
$$
 \left|B\backslash \left(\frac{1}{t} {\mathcal R}_{t_1}(0,\omega)\right)\right|\leq \delta\;.
$$
\end{lem}

\noindent {\bf Proof : }
Fix  
$\omega\in \Omega_0$, the set of full measure of $\omega$'s such that $B \subset {\mathcal E}^-$, 
let $t_n\to+\infty$ be any sequence and $\gamma\in (0,1)$ to be chosen later. Since, in view of Lemma \ref{lem:VolR}, there exist some constant $K>0$ and a sequence $s_n\in [\gamma t_n, t_n]$ such that
$$
{\rm Per}\left({\mathcal R}_{s_n}(0,\omega)\right) \leq K s_n^{N-1}\;,
$$
there exists a subsequence $(s_{n'})_{n' \in \N}$ such that the uniformly bounded sets ${\mathcal R}_{s_{n'}}(0,\omega)/s_{n'}$
converge in $L^1$, as $n' \to \infty$,  to some set $E$. In addition we may also assume that the bounded sequence $(s_{n'}/t_{n'})_{n'\in \N}$ converges to some $\eta\in [\gamma,1]$. We claim that
\be\label{claim:BmoinsE=0}
\left| B \backslash E\right|=0\;.
\ee
Assuming \eqref{claim:BmoinsE=0},  we complete the proof of the Lemma.  Since, as $n'\to\infty$,
\be\label{RtnRsn}
 \frac{1}{t_{n'}}{{\mathcal R}_{s_{n'}}(0,\omega)} \to \eta E \qquad {\rm in }\; L^1\;,
\ee
we find that
$$
\begin{array}{rl}
\lim_{n'\to\infty} \left|B \backslash \left(\frac{1}{t_{n'}}{{\mathcal R}_{s_{n'}}(0,\omega)}\right)\right|
=  & \left|B \backslash (\eta E)\right|  \\
\leq & \left|B \backslash B(0,\eta)\right| \; \leq \; \left|B \backslash B(0,\gamma)\right|=c_N(1-\gamma)^N\;.
\end{array}
$$

Accordingly we have proved that, for any sequence $t_n\to+\infty$, there exits a subsequence $(t_{n'})_{n'\in\N}$ and some $s_{n'}\in [\gamma t_{n'}, t_{n'}]$ such that, for sufficiently large $n'$,
$$
\left|B \backslash \left(\frac{1}{t_{n'}}{{\mathcal R}_{s_{n'}}(0,\omega)}\right)\right|\leq 2c_N(1-\gamma)^N,
$$
which gives the result with $k=(1/2c_N)^{1/N}$ and $\gamma =1-k\delta^{1/N}$. \\

We now show (\ref{claim:BmoinsE=0}). To simplify the notation, for the rest of the proof we write $s_n$ instead of $s_{n'}$. To this end, observe that it suffices to check that any $z\in  {\rm Int}(B)$ has a positive density in $E$, i.e., that, for any $z\in  {\rm Int}(B)$,
\be\label{PositiveDensity}
\liminf_{r\to 0^+} \frac{\left| E\cap B(z,r) \right| }{|B_r|}>0\;.
\ee
Fix $z\in {\rm Int}(B)$ and  $r>0$, choose $\sigma= r/(1+2M)$, and
note, for later use, that $B((1-\sigma)z, 2M\sigma)\subset B(z,r)$.

Lemma \ref{lem:B01subsetE-} yields $z_n\in {\mathcal R}_{(1-\sigma) s_n}(0,\omega)$ such that, as $n\to\infty$, $z_n/(1-\sigma)  s_n \to z$. Then ${\mathcal R}_{\sigma s_n}(z_n, \omega)\subset {\mathcal R}_{s_n}(0,\omega)$ while, for $n$ large enough,
$$
\frac{1}{s_n} {\mathcal R}_{\sigma s_n}(z_n, \omega)\subset \frac{1}{s_n} B(z_n,M\sigma s_n) \subset B((1-\sigma)z, 2M\sigma)\subset B(z,r)\;.
$$

Then (\ref{PositiveDensity}) follows since, for the $\beta$ defined in Lemma \ref{lem:VolR}, we have
$$
\begin{array}{rl}
\ds \left| E\cap B(z,r) \right| \; =  & \ds \lim_{n\to +\infty} \left| \frac{1}{s_n}{\mathcal R}_{s_n}(0,\omega) \cap B(z,r) \right|  \;
\geq \;  \liminf_{n\to +\infty} \left| \frac{1}{s_n}{\mathcal R}_{\sigma s_n}(z_n,\omega) \cap B(z,r) \right| \\
\geq & \ds \liminf_{n\to +\infty} \left| \frac{1}{s_n}{\mathcal R}_{\sigma s_n}(z_n,\omega) \right| \;
\geq \;  \beta \sigma^N \; = \;  \frac{\beta}{(1+2M)^N} r^N.
\end{array}
$$
\QED

Having established all the necessary ingredients,
we now proceed with the  \\

\noindent {\bf Proof of Theorem \ref{reach}: } Fix $\omega\in\Omega_0$, where $\Omega_0$ is the (finite) intersection of the sets of full measure where $V$ is defined and for which Lemmata \ref{lem:E+E-indepom} and \ref{lem:RtB01} hold, $\delta>0$ sufficiently small to be chosen later and let  $T_1=T_1(\omega,\delta)>0$ be the measurable map defined in Lemma \ref{lem:RtB01} so that, for any $t\geq T_1$, there is exists  $t_1\in [(1-k\delta^{1/N})t,t]$  such that
\be\label{B01Back}
\left|B \backslash \left(\frac{1}{t} {\mathcal R}_{t_1}(0,\omega)\right)\right|< \delta \;.
\ee

We first prove that, for every $\ep>0$,  there exists a (measurable in $\omega$) $T_2(\omega,\ep)>0$ such that, for all $ y\in \R^N $,
\be\label{Proofforx=0}
\theta(0,y,\omega) \leq T_2(\omega,\ep)+(1+\ep)|y| \;.
\ee

Fix $y\in \R^N$ and $\ep>0$. In view of Lemma \ref{lem:VolR}, for any $s\geq 0$  we have
$$
\left|{\mathcal R}_{-s}(y,\omega)\right|\geq \beta s^N \qquad {\rm and }\qquad {\mathcal R}_{-s}(y,\omega)\subset B(0,|y|+Ms)\;.
$$

Choose  next
$$s=\max\{(\frac{|y| \delta^{1/N}}{\beta^{1/N}-\delta^{1/N}M}), T_1(\omega,\delta)\left(\frac{\delta}{\beta}\right)^{1/N}\} \quad \text{ and } \quad
t=(\beta/\delta)^{1/N}s. $$

Since  $t\geq |y|+Ms$, 
we have
\be\label{B01Back2}
\frac{1}{t} {\mathcal R}_{-s}(y,\omega)\subset B
\ \  {\rm and }
\qquad
\left|\frac{1}{t} {\mathcal R}_{-s}(y,\omega)\right|\geq \frac{\beta s^N}{t^N} =\delta \;.
\ee

Observing that $t\geq T_1$ (a consequence of $s\geq T_1(\omega,\delta)(\delta/\beta)^{1/N}$), we conclude that there exists  $t_1\in [(1-k\delta^{1/N})t,t]$ such that (\ref{B01Back}) holds, a fact which, in view of (\ref{B01Back2}), implies that
$$
\frac{1}{t} {\mathcal R}_{t_1}(0,\omega) \cap \frac{1}{t} {\mathcal R}_{-s}(y,\omega) \neq\emptyset\;.
$$

Therefore $y\in {\mathcal R}_{s+t_1}(0,\omega)$ and, hence, since $t_1\leq t$,
$$
\theta(0,y,\omega)\leq s+t_1\leq  |y|\left(\frac{1+(\delta/\beta)^{1/N}}{1-(\delta/\beta)^{1/N}M}\right)+ (1+(\delta/\beta)^{1/N})T_1(\omega,\delta)\;.
$$
This gives (\ref{Proofforx=0}) for $\delta$ sufficiently small so that $1+(\delta/\beta)^{1/N}/(1-(\delta/\beta)^{1/N}M) \leq 1+\ep$ and
$T_2(\omega,\ep)=  (1+(\delta/\beta)^{1/N})T_1(\omega,\delta)$.\\

Next we prove the full reachability. Fix $\ep>0$, let  $k$ and $\delta_0$ be defined in Lemma \ref{ReachEom}, $\delta\in (0,\delta_0)$ be so small that
$k\delta^{1/N}(1+(1+\ep)M)<\ep$  and
$L$ sufficiently large so that the set $E_\ep=\{\omega\in \Omega_0 :
T_2(\omega,\ep)< L\}$ has probability larger then  $1-\delta$ with  $T_2(\omega,\ep)$ as above.
Let $\Omega_{0,\ep}$ denote the set of full probability obtained as the intersection of $\Omega_0$ with the set $\Omega'$ (of full probability) associated to $E_\ep$ by Lemma \ref{ReachEom}. Let $K=K(\omega,\ep)>0$ be defined by Lemma \ref{ReachEom} for each $\omega\in \Omega_{0,\ep}$.
Then, for all $x\in \R^N$ and for $t=t(\omega,\delta)={K(\omega)+k\delta^{1/N}|x|}$,
$
 {\mathcal R}_t(x,\omega) \cap E(\omega,\ep)\neq \emptyset 
$
where, as usual,  $E(\omega,\ep)=\{z\in \R^N :\tau_z\omega\in E_\ep \}$.  Choose $x_1\in  {\mathcal R}_t(x,\omega) \cap E(\omega,\ep)$.
Since $\tau_{x_1}\omega \in E_\ep$, we know from the first step that, for any $y\in\R^N$ there exists some
$t_1 \leq L+(1+\ep)|y-x_1|$ such that $y-x_1\in {\mathcal R}_{t_1}(0,\tau_{x_1}\omega)$ while, in view of
$
{\mathcal R}_{t_1}(x_1, \omega)=
{\mathcal R}_{t_1}(0, \tau_{x_1}\omega)+x_1\;,
$
$y\in {\mathcal R}_{t_1}(x_1, \omega)\subset  {\mathcal R}_{t+t_1}(x, \omega) $.
It follows  that, for some $T(\omega,\ep)>0$,
$$
\begin{array}{rl}
\theta(x,y,\omega)\leq t+t_1 \; \leq & t+L+(1+\ep)(|y-x|+ |x-x_1|)\\
\leq & L+K(1+(1+\ep) M)+ k\delta^{1/N}(1+(1+\ep) M)|x|+ (1+\ep)|y-x|\\
\leq & T(\omega,\ep)+ \ep|x|+(1+\ep)|y-x|.
\end{array}
$$

We conclude noting that we can remove the $\ep-$dependence of the set $\Omega_{0,\ep}$ by replacing it by  $\bigcap_n \Omega_{0,\ep_n}$ with  $\ep_n\to 0$.
\QED

\section{Homogenization}\label{sec:genN}

Here we show that the G-equation homogenizes. The proof relies on the next theorem which yields  the existence of the ``time constant''.

We have:

\begin{thm}[Existence of the ``time constant'']\label{thm:TimeConstant}
Assume (\ref{eq:ass1}, \ref{eq:ass2}) and (\ref{eq:mz}).  There exist a set of full probability $\Omega_0\subseteq \Omega$ and, for any $v\in \R^N$,  a ``time  constant'' \ $\bar q(v)\geq 0$ such that, for all $v\in\R^N$ and all $\omega\in \Omega_0$,
$$
\lim_{t\to +\infty} \frac{1}{t} \theta( 0, tv, \omega)= \bar q(v). 
$$
The map $v\to \bar q(v)$ is positively homogeneous, convex and such that, for $v\in\R^N$,
$$
M^{-1} |v| \leq \bar q(v) \leq |v|,
$$
and
\be \label{E+E-}
{\mathcal E}^+={\mathcal E}^-=\{v\in \R^N\;:\; \bar q(v)\leq 1\} \;.
\ee
Finally, for all $R>0$ and $\omega\in \Omega_0$,
\be\label{PointFinal}
\lim_{r\to+\infty} \sup_{|v|\leq R, \ |x|\leq R} \left| \frac{\theta(rx, r(x+v),\omega)}{r}-\bar q(v)\right|=0. 
\ee
\end{thm}

As a consequence, we have:

\noindent {\bf Proof of Theorem \ref{theo:hom1}: } If $V$ has mean zero, then, following \cite{S91} (see also \cite{NolenNovikov}), Theorem \ref{thm:TimeConstant} readily implies the homogenization property of the G-equation with effective  Hamiltonian $\bar H$ given by
$$
\overline H(p)= \sup_{ \bar q(v)\leq 1} \ \lg p, -v\rg \;.
$$
It is, of course, clear that $\overline H$ is positively homogeneous of degree one,
Lipschitz continuous and convex. Finally, since $\bar q(v)\leq |v|$ for any $v$, the vector $z=-p/|p|$ satisfies $\bar q(z)\leq 1$, so that $\bar H(p)\geq -\lg p, z\rg= |p|$.\\

We now assume that $\E[V]\neq 0$. For any $P\in \R^N$, let $z^P$ be the unique solution to
\be\label{eq:zP}
\left\{\begin{array}{l}
z_t= |Dz+P|+ \lg V, Dz+P\rg \ \text{ in } \ \R^N\times \R_+,\\[2mm]
z=0 \ \text{ on } \ \R^N\times\{0\}.
\end{array}\right.
\ee

It is well known (see \cite {LS, AS}) that the homogenization of (\ref{eq:uep}) is equivalent to the, a.s. in $\omega$ and uniform in balls of radius $Rt$ convergence, as $t\to+\infty$,  of $t^{-1} z^P(\cdot,t)$ to some constant $\overline H(P)$ and for all $P\in \R^N$ and $R>0$. 

Let now $\tilde V=V-\E[V]$. Since $\E[\tilde V]=0$, it follows from the first part of the ongoing proof that there exists a  homogenized Hamiltonian $\tilde H$ such that,  for all $P\in \R^N$, $\tilde H(P)\geq |P|$. 
It is also immediate that
$$
\tilde z^P(x,t)= z^P\left(x-\E[V] t, t\right)\ -\ \lg \E[V],P\rg \ t\;
$$
solves
$$
\left\{\begin{array}{l}
z_t= |Dz+P|+ \lg \tilde V, Dz+P\rg \ \text{ in } \ \R^N\times \R_+,\\[2mm]
z=0 \ \text{ on } \ \R^N\times\{0\}.
\end{array}\right.
$$

Since we have homogenization for $\tilde V$, for any $R>0$ and a.s. in $\omega$, we find
$$
\begin{array}{l}
\ds \lim_{t\to+\infty}\sup_{|x|\leq Rt}  \left|t^{-1}z^P(x,t)-(\tilde H(P)+\lg \E(V),P\rg)\right| \\
\qquad \qquad   =
  \ds  \lim_{t\to+\infty}\sup_{|x|\leq Rt}  \left|t^{-1}\tilde z^P(x+\E[V]t,t)-\tilde H(P)\right| \\
\qquad \qquad   \leq  \ds  \lim_{t\to+\infty}\sup_{|x|\leq (R+|\E[V]|)t}  \left|t^{-1}\tilde z^P(x,t)-\tilde H(P)\right|\; =\; 0
\end{array}
$$

It follows that the homogenization holds for $V$, with the homogenized Hamiltonian given by $\bar H(P)=\tilde H(P)+\lg \E(V),P\rg$. Moreover, for all $P\in \R^N$,
we have $\bar H(P)\geq |P|+\lg \E(V),P\rg$. 
\QED

The rest of the section is devoted to the proof of Theorem \ref{thm:TimeConstant}. For this we assume, that  (\ref{eq:ass1}, \ref{eq:ass2}) and (\ref{eq:mz}) hold. As already discussed,
the main difficulty stems from the lack of an integral bound for $\theta$, a fact which prevents us to use directly the sub-additive ergodic theorem. Our idea to overcome this difficulty is to construct, for each direction $\bar a\in B$, a random control $\alpha$ such that (i) the associated solution $X_t^{0,0,\alpha,\omega}$ of \eqref{eq:DS} has an a.s. long time average $\lim_{t\to\infty} X_t^\alpha/t$ a.s. close to $\bar a$, (ii) it is possible to apply the sub-additive ergodic theorem to the minimal time $\theta(X_s^\alpha, X_t^\alpha,\omega)$ to obtain a (random) a.s. limit, and (iii) the averaged long time limits of $\theta(0, X_t^\alpha,\omega)$ and $\theta(0, t\bar a,\omega)$ are a.s. close by the reachability estimate.  Since the ergodicity and stationarity yield that the largest and smallest possible limits, as $t\to\infty$, of $\theta(0, t\bar a,\omega)/t$ are actually almost surely independent of $\omega$ we may conclude.  The construction of the random controls is rather involved. When $N=2$, however, the special geometry allows to simply use constant controls, i.e., to take $\alpha=\bar a$. We present in the Appendix this simpler proof.


We begin with the last of the steps discussed above.

\begin{lem}\label{limsupliminf} There exists  $\Omega_0\subseteq\Omega$ of full probability such that, for all  $v\in \R^N$ and $\omega\in\Omega_0$, the limits
$
\limsup_{t\to+\infty} t^{-1}\theta(0,tv,\omega)$  and $\liminf_{t\to+\infty} t^{-1}\theta(0,tv,\omega)
$
are independent of $\omega$.
\end{lem}

\noindent {\bf Proof : } Since the arguments are similar, we only present the proof for the $\limsup$. Fix $\omega\in\Omega_0$, the latter been defined in the proof of Theorem \ref{reach}, and let $\ep>0$ be small. It follows that there exists $T=T(\omega,\ep)>0$ such that, for all $x,y\in \R^N$,
$$
\theta(x,y,\omega)\leq T(\omega,\ep) +\ep|x|+(1+\ep)|y-x|. 
$$
Then, for any $y\in \R^N$, we have
$$
\begin{array}{rl}
\theta(0,tv,\tau_y\omega)
= &
\theta(y,y+tv,\omega)\\
\leq  &
\theta(y,0,\omega)+\theta(0,tv,\omega)+\theta(tv,y+tv,\omega) \\
\leq &
T(\omega,\ep)+(1+2\ep) |y|+\theta(0,tv,\omega)+T(\omega,\ep)+\ep |tv|+(1+\ep)|y|,
\end{array}
$$
and, hence,
$$
\limsup_{t\to+\infty} \frac{1}{t}\theta(0,tv,\tau_y\omega)
\leq
\limsup_{t\to+\infty} \frac{1}{t}\theta(0,tv,\omega)
+\ep|v|\;,
$$
which proves the claim since $\ep$ is arbitrary.
\QED

We describe now the construction of the random control. We fix  $\bar a \in B$,  with $|\bar a|<1$, and we introduce some random variations around $\bar a$ in the following way. For $\ep\in (0, 1-|\bar a|)$ we set
$$
a_k = \left\{ \begin{array}{rl} \bar a + \ep e_k & {\rm for }\; k=1, \dots, N,\\[2mm]
\bar a -\ep e_{k-N} &  {\rm for }\; k=N+1, \dots, 2N,
\end{array}\right.
$$
where $\{e_k\}$ is the usual basis of $\R^N$.
Note that, since $a_k \in B$ for any $k$, the $a_k$'s are also admissible controls.

We denote by $Q$ the product probability measure on the set $D=(0,2\delta)\times \{1,\dots,2N\}$ given, for all $a,b\in (0,2\delta)$ such that $a<b$ and all $k\in \{1,\dots,2N\}$,
 by
$$
Q((a,b)\times \{k\}) = \frac{1}{2N}\int_{(a,b)}f(t)dt ,
$$
where $f:(0,2\delta)\to [0,1]$ satisfies $\int_0^{2\delta} f(t) dt =1$. This simply means that $t$ and $k$ are chosen independently with probability $f(t)$ and $1/2N$ respectively. For reasons which will become clear later in the section,
we choose the density $f$ so that
\be\label{eq:f}
\int_0^{2\delta}|t-\delta|f(t)dt \leq \delta^2 \ \  \text{ and } \ \  \int_0^{2\delta} t f(t) dt =\delta.
\ee

Having introduced the random setting, we now present the heuristic idea of the  construction of the random control $\alpha$. At time ${\bf \sigma}_0=0$, we choose at random (according to the probability $Q$) a pair $Z_0=({\bf t}_0, a_{{\bf k}_0})\in D$. We set $\alpha= a_{{\bf k}_0}$ on $ [0,{\bf t}_0]$ and ${\bf \sigma}_1= {\bf \sigma}_0+{\bf t}_0$. At time  ${\bf \sigma}_1$ we choose at random (again according to the probability $Q$) a new pair $Z_1=({\bf t}_1, a_{{\bf k}_1})\in D$, independent of $Z_0$, and set $\alpha= a_{{\bf k}_1}$ on $[{\bf \sigma}_1,{\bf \sigma}_1+{\bf t}_1]$ and ${\bf \sigma}_2= {\bf \sigma}_1+{\bf t}_1$. By induction this yields  almost surely with respect to $Q$ a control $\alpha : [0,+\infty)\to \{a_1, \dots, a_{2N}\}$ which is constant on each random interval $[{\bf \sigma}_i, {\bf \sigma}_{i+1}]$ of length at most $2\delta$. 

The first part of the proof consists in showing that, as $t\to\infty$, $X^\alpha_t/t$ is, a.s. in the product probability space, close to $\bar a$.
To prove this we first note that the system we are constructing is ergodic. Indeed, for $z=(t,k)\in D$, we consider the transformation $T^{z}\omega=\tau_{X^{0,0, a_k,\omega}_t}\omega$. It is clear  that $(T^z)_{z\in D}$ is a family of measure preserving transformations on $\Omega$. The following lemma states that this family  is ergodic.

\begin{lem}\label{lem:KakutaniErgo} If $E\in {\mathcal F}$ is invariant by $T^z$ for Q-almost every $z\in D$, then $\P(E)=0$ or $\P(E)=1$.
\end{lem}


We postpone the proof of Lemma~\ref{lem:KakutaniErgo}, which is rather technical, to the end of the section and we proceed with the (rigorous) construction of the control $\alpha$.

Let  $(Z_n=({\bf t}_n, {\bf k}_n))_{n\in \Z}$ be a sequence of i.i.d. random variables on $D$ with law $Q$ which is taken to be an element of the ``canonical'' probability space $(\overline \Omega, \overline  {\mathcal F}, \overline \P)$ with  $\bar \Omega= D^{\Z}$.
We denote by $\bar \tau : \overline \Omega\to \overline \Omega$ the usual right-shift, i.e., for $\bar \omega=( z_n)_{n\in \Z}\in \overline \Omega$, $\bar \tau\bar\omega= (z_{n+1})_{n\in \Z}$. By construction, $\bar \tau$ is a measure preserving transformation of $\overline \Omega$ and  the sequence $(Z_n)_{n\in \Z}$ is stationary with respect to $\bar \tau$. \\

Let $F\in L^1(\Omega)$. Since, in view of Lemma~\ref{lem:KakutaniErgo}, the probability space $(\Omega, {\mathcal F}, \P)$ is ergodic with respect to the family $(T^z)_{z\in D}$, the Kakutani random ergodic Theorem \cite{Ka51}  states that, in $L^1(\Omega\times{\overline \Omega})$ and a.s. with respect to the product measure $\P\otimes \bar \P$,
\be\label{RandomErgoTheo}
\lim_{n\to +\infty} \frac{1}{n} \sum_{i=0}^{n-1} F\left( T^{Z_i(\bar \omega)}\circ \dots \circ T^{Z_0(\bar \omega)}\omega\right) =\E[F].
\ee

As described above, the random control $\alpha=\alpha(\bar \omega)$ is given by
\be\label{eq:rc}
\alpha(\bar \omega)(t)= a_{{\bf k}_n(\bar \omega)} \; {\rm for }\; \ t\in [\sigma_n(\bar \omega),\sigma_{n+1}(\bar \omega))
\quad {\rm where }\quad
{\bf \sigma}_n(\bar \omega)= \sum_{i=0}^n {\bf t}_i(\bar \omega)\;.
\ee

A simple induction argument yields
$$
T^{Z_i(\bar \omega)}\circ \dots \circ T^{Z_0(\bar \omega)}\omega=
\tau_{X^{0,0,\alpha(\bar \omega),\omega}_{{\bf \sigma}_i(\bar \omega)}}\omega.
$$

Hence, applying (\ref{RandomErgoTheo}) to the map $V$, we find that, $\P\otimes \bar \P$ a.s. and in $L^1(\Omega\times{\overline \Omega})$,
\be\label{CvVX}
\lim_{n\to +\infty} \frac{1}{n} \sum_{i=0}^{n-1} V\left( X^{0,0,\alpha(\bar \omega),\omega}_{{\bf \sigma}_i(\bar \omega)}, \omega\right) =\E[V]=0\;,
\ee
a fact that implies, as we explain next, that the sequence $X^{0,0,\alpha(\bar \omega), \omega}_t$ is asymptotically close to the line $\R\bar a$.

We have:

\begin{lem}\label{lem:limXovert}
Let $\alpha \in {\mathcal A}$ be defined by \eqref{eq:rc}. There exists $C>0$ depending only on $\|V\|_{{\mathcal C}^{1,1}}$ such that, a.s. in   $\P\otimes \bar \P$,
\be\label{limXovert}
\limsup_{t\to+\infty} \left|\frac{1}{t} X^{0,0,\alpha(\bar \omega), \omega}_t-\bar a\right| \leq C\delta. 
\ee
\end{lem}

\noindent {\bf Proof : } To simplify the notation, for the rest of this proof we omit $\omega$ and $\bar \omega$ and we simply write set $X^{\alpha}_{{\bf \sigma}_n}$ in place of
$X^{0,0,\alpha(\bar \omega),\omega}_{{\bf \sigma}_n(\bar \omega)}$.

Since $V$ is Lipschitz continuous and ${\bf t}_i\leq \delta$,
$$
\left| X^{\alpha}_{{\bf \sigma}_{i+1}}-
X^{\alpha}_{{\bf \sigma}_i }- {\bf t}_i \left(V(X^{\alpha}_{{\bf \sigma}_{i}})+ a_{{\bf k}_i}\right)\right| \leq C {\bf t}_i^2\leq C \delta^2\;
$$
and
$$
\frac{1}{{\bf \sigma}_n} X^{\alpha}_{{\bf \sigma}_n}
= \frac{1}{{\bf \sigma}_n} \sum_{i=0}^{n-1} \left(X^{\alpha}_{{\bf \sigma}_{i+1}}-
X^{\alpha}_{{\bf \sigma}_i}\right)
= \frac{1}{{\bf \sigma}_n} \sum_{i=0}^{n-1} {\bf t}_i \left(V(X^{\alpha}_{{\bf \sigma}_{i}})+ a_{{\bf k}_i}\right) +O(\frac{n\delta^2}{{\bf \sigma}_n}).
$$

Therefore
\be\label{XoverTaun}
\frac{1}{{\bf \sigma}_n} X^{\alpha}_{{\bf \sigma}_n}
= \frac{\delta}{{\bf \sigma}_n} \sum_{i=0}^{n-1}  V(X^{\alpha}_{{\bf \sigma}_{i}})
+\frac{1}{{\bf \sigma}_n} \sum_{i=0}^{n-1} ({\bf t}_i-\delta) V(X^{\alpha}_{{\bf \sigma}_{i}})+\frac{1}{{\bf \sigma}_n} \sum_{i=0}^{n-1}  {\bf t}_ia_{{\bf k}_i}
+O(\frac{n\delta^2}{{\bf \sigma}_n}).
\ee

Next we observe that, as $n\to\infty$, the law of large numbers yields that, $\overline \P$ a.s.,
\be\label{LLN}
\frac{{\bf \sigma}_n}{n} \to  \E[{\bf t}_0]=  \delta\;, \;
\frac{1}{n} \sum_{i=0}^{n-1}  {\bf t}_ia_{{\bf k}_i} \to \delta \bar a\; \ {\rm and} \ 
\frac{1}{n} \sum_{i=0}^{n-1} |{\bf t}_i-\delta|\to \E[|{\bf t}_0-\delta|]\leq \delta^2, 
\ee
where $\bar \E$ denotes the expectation with respect to the probability measure $\bar \P$.

Combining (\ref{CvVX}), (\ref{XoverTaun}) and (\ref{LLN}), we get, $\P\times\overline\P$ a.s., the estimate
$$
\limsup_{n\to\infty} \left|\frac{1}{{\bf \sigma}_n} X^{\alpha}_{{\bf \sigma}_n}-\bar a\right| \leq C\delta.
$$

Finally, using that (i)  for any $t\geq0$, there is some $n$ such that ${\bf \sigma}_n\leq t\leq {\bf \sigma}_{n+1}$ with ${\bf \sigma}_{n+1}-
{\bf \sigma}_{n}\leq 2\delta$, and (ii) the dynamics are bounded, we obtain
(\ref{limXovert}).
\QED

The next lemma is about the averaged long time  behavior of the minimal time along the special trajectory  $X^{0,0,\alpha(\bar \omega), \omega}_{\sigma_n(\bar \omega)}$.

We have:

\begin{lem}\label{lem:Gamma}
There exists a random variable $\Gamma:\Omega\times\overline\Omega\to\R$ such that the averaged minimal time $n^{-1}\theta(0,
X^{0,0,\alpha(\bar \omega), \omega}_{\sigma_n(\bar \omega)}, \omega)$ converges, as $n\to\infty$ and  a.s. in  $(\omega, \bar \omega)$, to $\Gamma(\omega,\bar \omega)$.
\end{lem}

\noindent {\bf Proof : }  
We introduce on $\Omega\times \bar \Omega$ the measure preserving transformation
$$
\tilde T(\omega, \bar \omega)= \left( T^{a_{{\bf k}_0}(\bar \omega)}_{{\bf t}_0(\bar \omega)}\omega , \bar \tau \bar \omega\right),
$$
and note that, in view of \eqref{eq:rc}, for any $n\in \N$,
$$
\tilde T^n (\omega, \bar \omega)= \left( \tau_{X^{0,0, \alpha(\bar \omega), \omega}_{\sigma_n(\bar \omega)}}\omega, \bar \tau^n \bar \omega\right).
$$

Consider the family of random variables
$$
({\mathcal Z}(n,\omega,\bar \omega))_{n\in\N}= (\theta(0,X^{0,0,\alpha(\bar \omega),\omega}_{\sigma_n(\bar \omega)},\omega))_{n\in\N}.
$$

The first observation is that each ${\mathcal Z}(n,\omega,\bar \omega)$ is bounded and, hence, integrable. Indeed, since $X^{0,0,\alpha(\bar \omega),\omega}_t$ is a solution of \eqref{eq:DS}, we have
$$
{\mathcal Z}(n,\omega,\bar \omega)\leq \sigma_n(\bar \omega)\leq 2\delta n\;.
$$

It also turns  out that the family is sub-additive, i.e., for all $n,m\in\N$ and $(\omega,\bar \omega)\in\Omega\times\bar  \Omega$,
\be\label{Zsubadditive}
{\mathcal Z}(n+m,\omega,\bar \omega)\leq {\mathcal Z}(n,\omega,\bar \omega)+{\mathcal Z}(m,\tilde T^n(\omega,\bar \omega)).
\ee

To see this we first remark that we clearly have
$$
{\mathcal Z}(n+m,\omega,\bar \omega)\leq {\mathcal Z}(n,\omega,\bar \omega)+\theta(X^{0,0,\alpha(\bar \omega),\omega}_{\sigma_n(\bar \omega)},
X^{0,0,\alpha(\bar \omega),\omega}_{\sigma_{n+m}(\bar \omega)}, \omega).
$$

Since, for any $n\in \N$,
$$
\sigma_{m+n}(\bar \omega)= \sigma_n(\bar \omega)+\sigma_m(\bar \tau^n\bar \omega)\;,
$$
the semi-group property of the flow yields
$$
X^{0,0,\alpha(\bar \omega),\omega}_{\sigma_{m+n}(\bar \omega)}
=
X^{\sigma_n(\bar \omega),X^{0,0,\alpha(\bar \omega),\omega}_{\sigma_{n}(\bar \omega)},\ \alpha(\bar \omega),\ \omega}_{\sigma_n(\bar \omega)+\sigma_m(\bar \tau^n\bar \omega)}
=
X^{0,X^{0,0,\alpha(\bar \omega),\omega}_{\sigma_{n}(\bar \omega)},\ \alpha(\bar \omega)(\cdot+\sigma_n(\bar \omega)),\ \omega}_{\sigma_m(\bar \tau^n\bar \omega)}\;.
$$

The definitions of the random control $\alpha$ and the random time ${\bf \sigma}_n$ also imply that
$$
\alpha(\bar \omega)(\cdot+\sigma_n(\bar \omega))
=
\alpha(\bar \tau^n\bar \omega)(\cdot)\;,
$$
and, thus, 
$$
X^{0,0,\alpha(\bar \omega),\omega}_{\sigma_{m+n}(\bar \omega)}
=
X^{0,X^{0,0,\alpha(\bar \omega),\omega}_{\sigma_{n}(\bar \omega)},\alpha(\bar \tau^n\bar \omega),\omega}_{\sigma_m(\bar \tau^n\bar \omega)}
=
X^{0,0,\alpha(\bar \tau^n\bar \omega),\tau_{X^{0,0,\alpha(\bar \omega),\omega}_{\sigma_{n}(\bar \omega)}}\omega}_{\sigma_m(\bar \tau^n\bar \omega)}
+ X^{0,0,\ \alpha(\bar \omega),\omega}_{\sigma_{n}(\bar \omega)}\;.
$$

Set $Y^{\omega,\bar \omega}= X^{0,0,\alpha(\bar \omega),\omega}_{\sigma_{n}(\bar \omega)}$. Since, for any $x,y,z\in \R^N$,
$$
\theta(x+z,y+z,\omega)= \theta(x,y,\tau_z\omega)
$$
we obtain \eqref{Zsubadditive} from the following string of equalities:
\begin{align*}
\theta(X^{0,0,\alpha(\bar \omega),\omega}_{\sigma_n(\bar \omega)},
X^{0,0,\alpha(\bar \omega),\omega}_{\sigma_{n+m}(\bar \omega)}, \omega)\; = & \;
\theta \left(Y^{\omega,\bar \omega},
X^{0,0,\alpha(\bar \tau^n\bar \omega),\tau_{Y^{\omega,\bar \omega}}\omega}_{\sigma_n(\bar \tau^n\bar \omega)}+Y^{\omega,\bar \omega},
\omega\right) \\
= & \;
\theta \left(0,
X^{0,0,\alpha(\bar \tau^n\bar \omega),\tau_{Y^{\omega,\bar \omega}}\omega}_{\sigma_n(\bar \tau^n\bar \omega)},
\tau_{Y^{\omega,\bar \omega}}\omega\right) \\
=&\; {\mathcal Z}(m,\tau_{Y^{\omega,\bar \omega}}\omega,\bar \tau^n\bar \omega) \;
=  \; {\mathcal Z}(m,\tilde T^n(\omega,\bar \omega)).
\end{align*}

It follows now from Kingman's sub-additive ergodic theorem (see \cite{St89}) that there exists  a random variable $\Gamma:\Omega\times\bar \Omega\to\R$ such that, a.s. in $(\omega,\bar \omega)$,
\be\label{defGamma}
\lim_{n\to\infty} n^{-1}{\mathcal Z}(n,\omega,\bar \omega) =\Gamma(\omega,\bar \omega). 
\ee
\QED

Notice that in \eqref{defGamma} the long-time averaged limit $\Gamma$ is not deterministic, since the underlying measure preserving transformation is not itself ergodic. On the other hand, in view of the reachability estimate and Lemma~\ref{limsupliminf}, the existence of $\Gamma$ is enough to yield the existence of an a.s. in $\omega$ deterministic limit for $t^{-1} \theta(0, t\bar a)$.

We have:

\begin{lem}\label{lem:timeconst} There exist $\Omega_0\subseteq\Omega$ of full probability and a positively homogeneous of degree $1$ map $\overline q:\R^N\to \R_+$ such that, for all $v\in\R^N$,
\be\label{EstiBarq}
M^{-1} |v| \leq \bar q(v)\leq |v| \;,
\ee
and, for all $\bar a \in\R^N$ and $\omega\in\Omega_0$,  $t^{-1}\theta(0, t\bar a,\omega)$ converges, as $t\to \infty$,   to  $\overline  q(\bar a)$.
\end{lem}

\noindent {\bf Proof : }  Theorem \ref{reach} yields the existence, a.s. in $\omega\in \Omega$, of a constant $T(\omega, \ep)$ such that, for all $x,y\in \R^N$,
$$
\theta(x,y,\omega) \leq T(\omega,\ep)+\ep|x|+(1+\ep)|y-x|.
$$

To simplify the presentation, for the remainder of the proof we omit the dependence with respect to $(\omega,\bar \omega)$ and revert  to writing
$X^\alpha_t$ for  $X^{0,0,\alpha(\bar \omega),\omega}_t$. \\

Fix $\bar a \in B$. For $t>0$ large, let $n$ be such that ${\bf \sigma}_n\leq t < {\bf \sigma}_{n+1}$ and note that, since ${\bf \sigma}_n/n\to \delta$ a.s. in $\bar \omega$  as $n\to\infty$,
we also have $\lim_{n\to\infty}t/n\to \delta$ a.s. in $\bar \omega$. In addition, in view of the inequalities
\begin{align*}
\theta( 0, t\bar a) \leq & \;  \theta(0, X^\alpha_{{\bf \sigma}_n})
+ \theta( X^\alpha_{{\bf \sigma}_n}, t\bar a)\\
\leq & \;   \theta(0, X^\alpha_{{\bf \sigma}_n})
+ T(\omega,\ep)+ \ep|X^\alpha_{{\bf \sigma}_n}|+ (1+\ep)| X^\alpha_{{\bf \sigma}_n}-t\bar a|,
\end{align*}
using Lemma~\ref{lem:limXovert} and  Lemma~\ref{lem:Gamma}, we get
\begin{align*}
\limsup_{t\to+\infty}\frac{1}{t}  \theta( 0, t\bar a) \leq & \;
\limsup_{t\to +\infty} \frac{n}{t}\frac{1}{n}  \theta(0, X^\alpha_{{\bf \sigma}_n})
+ \ep(\|V\|_\infty+1)+ (1+\ep)\limsup_{t\to+\infty} \frac{1}{t} | X^\alpha_{{\bf \sigma}_n}-t\bar a| \\
\leq & \; \frac{1}{\delta}\Gamma
+ \ep(\|V\|_\infty+1)+ (1+\ep)C\delta,
\end{align*}
and, similarly,
\begin{align*}
\liminf_{t\to+\infty}\frac{1}{t}  \theta( 0, t\bar a) \geq & \;
\frac{1}{\delta}\Gamma
- \ep(\|V\|_\infty+1)- (1+\ep)C\delta\;.
\end{align*}

In particular we have
$$
\limsup_{t\to+\infty}\frac{1}{t}  \theta( 0, t\bar a)
- \liminf_{t\to+\infty}\frac{1}{t}  \theta( 0, t\bar a)
\leq
 2\ep(\|V\|_\infty+1)+ 2(1+\ep)C\delta\;,
 $$
which, since $\ep$ and $\delta$ are arbitrary, implies that $t^{-1}\theta(0,t \bar a,\omega)$ has, as $t\to\infty$, an a.s. limit $\overline q(\bar a)$. The fact that this limit is independent of $\omega$ has been proved in Lemma \ref{limsupliminf}.

It is also immediate that  $\bar q$ can be extended to a positively homogeneous of degree one map from $\R^N$ to $\overline \R_+$. To prove
(\ref{EstiBarq}), we first note that, in view  of the assumed bounds on $V$, we have, for all $t\geq 0$ and all controls $\alpha \in {\mathcal A}$,  $|X^{\alpha}_t|\leq M$. Hence, for all $v\in \R^N$ and $t\geq 0$,
$$
\theta(0, tv, \omega)\geq M^{-1} t|v| \;,
$$
and, therefore, $\bar q(v)\geq M^{-1}|v|.$ 

For the upper bound, we recall that  Theorem~\ref{reach} yields, a.s. in $\omega$ and for sufficiently small $\ep>0$, some positive $T(\omega,\ep)$, such that,
for all $v\in \R^N$ and $t\geq 0$,
$$
\theta(0, tv, \omega)\leq T(\omega,\ep)+(1+\ep)t|v|\;,
$$
and, hence, $\overline q(v) \leq (1+\ep)|v|$ for all $v\in\R^N$, which gives the claimed  upper bound since $\ep$ is arbitrary.
\QED

We are now in position to present the

\noindent {\bf Proof of Theorem \ref{thm:TimeConstant}: } Since all the statements below are true a.s. in $\omega$ throughout the proof we do not make any reference to this fact and, when we use an $\omega$, it is assumed that it belongs to the ``good'' subset of $\Omega$.

Fix $v\in \R^N\backslash \{0\}$. The definition of $\theta$ yields
$rv\in {\mathcal R}_{\theta(0,rv,\omega)}(0,\omega)$. Since
$\theta(0,rv,\omega)\to +\infty$ as $r\to+\infty$ and a.s., using Lemma~\ref{lem:timeconst}  and the convexity of ${\mathcal E}^-$, we find
$$
\frac{v}{\overline q(v)} = \lim_{r\to+\infty}\frac{rv}{\theta(0,rv,\omega)}
\in {\mathcal E}^- \ \ \text{ and } \ \
\{v \in \R^N : \bar q(v)\leq 1\} \subset {\mathcal E}^-\;.
$$

Conversely, if $z\in {\mathcal E}^+$, then there exists a sequence $t_n\to \infty$ and
$x_n\in {\mathcal R}_{t_n}(0,\omega)$ such that $x_n/t_n\to z$ and
$
\theta(0,x_n, \omega)\leq t_n\;.
$

Theorem \ref{reach} yields that,  for any $\ep>0$ and an appropriate $T(\ep,\omega)>0$,
$$
\theta(0,t_n z, \omega) \leq \theta(0,x_n, \omega)+\theta(x_n,t_nz,  \omega)
\leq t_n+ T(\ep,\omega)+\ep t_n|z| +(1+\ep)|x_n-t_nz|\;.
$$

Dividing by $t_n$ and letting first $n\to+\infty$ and then $\ep\to 0$ gives
$\bar q(z)\leq 1$, and, hence,
$$
{\mathcal E}^+\subset \{v\;;\; \bar q(v)\leq 1\}  \subset {\mathcal E}^-\;,
$$
and (\ref{E+E-}) holds.

Since $\overline q$ is positively homogeneous of degree $1$ and the set $ {\mathcal E}^-$ is convex, (\ref{E+E-}) implies that $\bar q$ is also convex.
Theorem \ref{reach} again yields (we do not repeat here the qualifiers) that, for any $v,v'\in \R^N$,
 $$
 \theta(0,tv, \omega)\leq  \theta(0,tv', \omega)+  \theta(tv',tv \omega)
\leq \theta(0,tv', \omega)+ T(\ep,\omega)+\ep t|v'|+(1+\ep)|v'-v|.
$$

The inequality above has the following two consequences. Firstly, the convergence of $\theta(0,tv,\omega)/t$ to $\overline q(v)$ is uniform with respect to $v$ for $|v|\leq R$. Secondly, after dividing by $t$ and passing to the limit $t\to\infty$, it follows that  $\overline  q$ is Lipschitz continuous with Lipschitz constant $1$.

We conclude showing (\ref{PointFinal}). We have just proved that, for each $\ep>0$ small, there exists a sufficiently large $T_\ep$ so that the set
$$
E_\ep= \left\{ \omega\in \Omega\;:\; \sup_{r\geq T_\ep, |v|\leq R} \left| \frac{\theta(0, rv,\omega)}{r}-\bar q(v)\right|\leq \ep
\right\}
$$
has a probability larger than $1-\ep$. 

In view of Lemma~\ref{ReachEom}, we know that there exist $k$ and a $\Omega_\ep\subseteq\Omega$ of full probability such that, for each $\omega\in\Omega_\ep$, there exists a positive constant $K(\omega,\ep)$ such that, for $t= K(\omega,\ep)+k\ep^{1/N} |x|$,
$$
{\mathcal R}_t(x,\omega)\cap E_\ep(\omega)\neq \emptyset ,
$$
where, as usual, $E_\ep(\omega)= \{x\in \R^N : \tau_x\omega\in {\mathcal E}_\ep\}\;.$

Fix $|v|, |x|\leq R$ and let $y\in {\mathcal R}_t(rx,\omega)\cap E_\ep(\omega)$ with $t$ as above. Then 
$$
\begin{array}{rl}
\theta(rx,r(x+v),\omega)
\leq &
\theta(y,r(x+v),\omega)
+\theta(rx,y,\omega)  \\
\leq & \theta(0,rx-y+rv,\tau_y \omega)
+T(\omega)+\ep r|x|+(1+\ep)|y-rx|\\
\leq & \theta(0,rv,\tau_y \omega)
+2T(\omega)+\ep rR+2(1+\ep)|y-rx|,
\end{array}
$$
where
$$
 |y-rx|\leq M t \leq MK(\omega)+ Mk\ep^{1/N}r |x|
 \leq  MK(\omega)+ MRk\ep^{1/N}r\;.
$$

Since $\tau_y\omega\in E_\ep(\omega)$ we have, for some positive constant $C(\omega)$ depending on $R$ but not on $x$ and $v$, that  
$$
\frac{\theta(rx,r(x+v),\omega) }{r}
\leq \bar q(v)+\ep + \frac{C(\omega)}{r} + C(\omega)\ep^{1/N},
$$
while a similar argument gives the reverse inequality, i.e,
$$
\frac{\theta(rx,r(x+v),\omega) }{r}
\geq \bar q(v)-\ep - \frac{C(\omega)}{r} - C(\omega)\ep^{1/N}\;.
$$

The conclusion, i.e., (\ref{PointFinal}), now follows after letting $\ep\to0$ providing we remove the $\ep$-dependence of the set $\Omega_\ep$ as in the proof of Theorem~\ref{reach} 
\QED

We now turn to the proof of  Lemma~\ref{lem:KakutaniErgo} which is rather involved. For the convenience of the reader we present first a heuristic argument. The main idea is that, since $E$ is invariant by $T^z$ for Q-almost every $z\in D$, the indicator function $u={\bf 1}_E$ must satisfy
\be\label{eq:Goal1}
0\geq \max_{i} \lg Du, V+ a_i\rg,
\ee
where
\be\label{Goal}
 \max_{i} \lg Du, V+ a_i\rg = \ep |Du|_\infty+ \lg V+\bar a, Du\rg .
 \ee

Since the vector field $V+\bar a$ is divergence free, then (\ref{eq:Goal1}) and \eqref{Goal} then imply
$$
0 \geq \ep \E[ |Du|_\infty]\;,
$$
which,  in turn, yields that $u$ is constant. Hence $\P[E]=0$ or $\P[E]=1$. Unfortunately the derivation of (\ref{eq:Goal1}) is not easy. The information we have is only a measure theoretic invariance for $E$ (recall $T^zE=E \ \P$  a.s. for almost all $z$), while the proof of an inequality like (\ref{eq:Goal1}) requires more ``point-wise'' information.  Getting around this difficulty is the main point of the proof and requires several steps.

Next we introduce some additional notation and summarize some facts that are used in the proof. To this end,
let $D_0\subset D$ be the subset of probability $1$ of $Q$ such that $E$ is invariant with respect to $T^z$ for all $z\in D_0$. Note that $D_0$ is dense in $D$. We restrict further $D_0$ by selecting in $D_0$ a countable dense sequence $D_1$. Recall that the invariance assumption means that
\be\label{PTEE}
\P[ T^z (E)\Delta E]=0\qquad \mbox{\rm for all $z\in D_1$.}
\ee

For any $\omega\in \Omega$ and any measurable set $S\subset \Omega$, we set as usual $S(\omega)=\{x\in \R^N : \tau_x\omega\in S\}$. Recall that, if $\P[S]=0$, then a simple application of Fubini's theorem yields the set of $\omega$'s, such that $S(\omega)$ has zero Lebesgue measure in $\R^N$, has probability $1$. So, in view of (\ref{PTEE}) and the ergodic theorem, there exists $\Omega_0\subseteq\Omega$ of full measure such that, for all $\omega\in \Omega_0$ and $z\in D_1$,
\be\label{ChoiceOmega}
\lim_{R\to\infty} \frac{|B_R\cap E(\omega)|}{|B_R|}=\P(E)\qquad{\rm and }\qquad \left| (T^z(E)\Delta E)(\omega)\right|=0.
\ee

Let $A=\{v\in\R^N : |v-\bar a|_\infty\leq \ep\}$ -- note that $A$ is the convex envelope of the set $\{a_k: k=1,\dots, 2N\}$, 
and denote by ${\mathcal A}_0\subset{\mathcal A}$ the set of time-measurable controls $\alpha: [0,+\infty)\to A$. Given $\alpha\in {\mathcal A}_0$, we denote by $X^\alpha_t$ the map $x\to X^{0,x,\alpha,\omega}_t$ and recall that, for each fixed $t$, it is a bi-Lipschitz continuous, and, hence, a bi-measurable, bijection from $\R^N$ to $\R^N$.\\

We are now ready for the

\noindent {\bf Proof of Lemma \ref{lem:KakutaniErgo}: } In view of the previous discussion we assume \eqref{PTEE} and $\P[E]>0$ and we show that $\P[E]=1$.
For the convenience of the reader we divide the proof in four steps.

{\bf Step 1:} We claim that, for all $\omega\in\Omega_0, t\in (0,\infty)$ and $\alpha\in {\mathcal A}_0$
\be\label{step1}
\left| X^{\alpha}_t(E(\omega))\Delta E(\omega)\right|=0.
\ee

If there is some $k\in\{1,\dots,2N\}$ such that $\alpha=a_k$ on $[0,t]$ and $z=(t,a_k)\in D_1$, then (\ref{step1}) holds  because, by the definition of $T^z$,
$$
\left|X^{\alpha}_t(E(\omega))\Delta E(\omega)\right|= \left|(T^zE\cap E)(\omega)\right|=0\;,
$$
the last equality coming from the choice of $\omega$ in (\ref{ChoiceOmega}).
By induction, (\ref{step1}) also holds  if there exist  sequences of times
$0=t_0<t_1<\dots< t_n=t$ and integers $k_1, k_2, \dots, k_n\in \{1, \dots, 2N\}$ with $(t_{i+1}-t_i, k_i)\in D_1$ such that  $\alpha= a_{k_i}$ on $[t_i,t_{i+1})$. For future reference we name such a control a simple control and say that $t$ is its associated final time.

Next we fix $\alpha\in {\mathcal A}_0$ and $t>0$. Since $D_1$ is dense in $D$, there exists a sequence of simple controls $(\alpha_n)_{n\in\N}$ with associated final times $t_n$ such that, as $n\to\infty$,  $\alpha_n \to \alpha$ in $L^\infty_{loc}$-weak-$\ast$ and $t_n\to t$. The goal  is to pass to the limit in the equality $\left| X^{\alpha_n}_{t_n}(E(\omega))\Delta E(\omega)\right|=0$.  Note that, in view of the weak convergence of the $\alpha_n$'s to $\alpha$, the map $x\to X^{\alpha_n}_{t_n}(x)$ converges locally uniformly to the map $x\to X^{\alpha}_{t}(x)$. We denote by $Y^{\alpha_n}_{t_n}$ and $Y^{\alpha}_{t}$ the inverse maps of $X^{\alpha_n}_{t_n}$ and $X^{\alpha}_{t}$ respectively. For any sufficiently large $R>0$ and sufficiently small $\eta>0$, let $\phi:\R^N\to [0,1]$ be a smooth map  with compact support in $B_R$ such that $\|{\bf 1}_{E(\omega)\cap B_R}-\phi\|_{L^1(\R^N)}\leq \eta$. Since $Y^{\alpha_n}_{t_n}$ and $Y^{\alpha}_{t}$ preserve the measure in $\R^N$, we also have
\be\label{coconuts}
\|{\bf 1}_{E(\omega)\cap B_R}\circ Y^{\alpha_n}_{t_n}-\phi\circ Y^{\alpha_n}_{t_n}\|_{L^1(\R^N)}\leq \eta\; {\rm  and }\;
\|{\bf 1}_{E(\omega)\cap B_R}\circ Y^{\alpha}_{t}-\phi\circ Y^{\alpha}_{t}\|_{L^1(\R^N)}\leq \eta\;.
\ee

In view of the local uniform convergence of the $X^{\alpha_n}_{t_n}$'s to $X^{\alpha}_{t}$ and the continuity of $\phi$, for $n\geq n_0$, where $n_0$ is sufficiently large,  we have $\| \phi\circ Y^{\alpha_n}_{t_n}- \phi\circ Y^{\alpha}_{t}\|_1\leq \eta$. This inequality combined with (\ref{coconuts}) implies that, for all $ n\geq n_0$,
 $$
 \|{\bf 1}_{E(\omega)\cap B_R}\circ Y^{\alpha_n}_{t_n}-{\bf 1}_{E(\omega)\cap B_R}\circ Y^{\alpha}_{t}\|_{L^1(\R^N)}\leq 3\eta.
 $$

Since $${\bf 1}_{E(\omega)\cap B_R}\circ Y^{\alpha_n}_{t_n}={\bf 1}_{X^{\alpha_n}_{t_n}(E(\omega)\cap B_R)} \ \text{ and } \
 {\bf 1}_{E(\omega)\cap B_R}\circ Y^{\alpha}_{t}={\bf 1}_{X^{\alpha}_{t}(E(\omega)\cap B_R)},$$
 we have shown that
\be\label{toto1}
\lim_{n\to +\infty} \|{\bf 1}_{X^{\alpha_n}_{t_n}(E(\omega)\cap B_R)}-{\bf 1}_{X^{\alpha}_{t}(E(\omega)\cap B_R)}\|_{L^1(\R^N)}= 0\;.
\ee

Recalling that $X^{\alpha_n}_{t_n}$ is a bijection and $\left| X^{\alpha_n}_{t_n}(E(\omega))\Delta E(\omega)\right|=0$, we also have, that, as $n\to\infty$,
$$
{\bf 1}_{X^{\alpha_n}_{t_n}(E(\omega)\cap B_R)} = {\bf 1}_{X^{\alpha_n}_{t_n}(E(\omega))\cap X^{\alpha_n}_{t_n}( B_R)}
= {\bf 1}_{E(\omega)\cap X^{\alpha_n}_{t_n}( B_R)} \to {\bf 1}_{E(\omega)\cap X^{\alpha}_{t}( B_R)} \; {\rm in } \; L^1(\R^N)\;.
$$

Combining (\ref{toto1})
and the claim above implies that
$${\bf 1}_{X^{\alpha}_{t}(E(\omega)\cap B_R)}=  {\bf 1}_{E(\omega)\cap X^{\alpha}_{t}( B_R)} \ \ \text{a.e. in} \ \  \R^N$$
and, hence, as  $R\to+\infty$, (\ref{step1}). \\

{\bf Step 2:} Let $\hat E(\omega)$ be the set of positive density points of $E(\omega)$, i.e.,
$$
\hat E(\omega)=\left\{ x\in \R^N\;: \; \liminf_{r\to 0^+} \frac{| B(x,r)\cap E(\omega)|}{|B(r)|} >0\right\}\;,
$$
and recall that, in view of the assumption at the beginning of the proof, we have $\hat E(\omega)=E(\omega)$ a.e. in $\R^N$.
We claim that $\hat E(\omega)$ is invariant under the action of $X^\alpha_t$ for any $\alpha\in{\mathcal A}_0$ and any $t\geq0$, i.e.,
\be\label{step2}
{\rm if }\; x\in \hat E(\omega), \quad {\rm then} \quad X^{\alpha}_t(x)\in \hat E(\omega)\;.
\ee
Note the difference with (\ref{step1}). Here we have a point-wise statement. \\

Since $X^\alpha_t$ is a bi-Lipschitz bijection of $\R^N$ with  Lipschitz constant $L$, for any $x\in \R^N$ and $\ep>0$, we have
$$
X^\alpha_t(B(x,\ep))\subset B(X^\alpha_t(x), L\ep).\;
$$

Let  $x\in  \hat E(\omega)$ and set $y= X^\alpha_t(x)$. Then, for any $r>0$,
$$
| E(\omega)\cap B(y,r)|\geq  | X^\alpha_t (E(\omega))\cap X^\alpha_t(B(x,r/L)|
=  | X^\alpha_t (E(\omega)\cap B(x,r/L))| = |E(\omega)\cap B(x,r/L)|,
$$
where the last equality holds because $X^\alpha_t$ preserves the measure.

Since $x$ is a point of positive density of $E(\omega)$, we have
$$
\liminf_{r\to 0^+} \frac{| E(\omega)\cap B(y,r)|}{|B(r)|} \geq
\liminf_{r\to 0^+} \frac{ |E(\omega)\cap B(x,r/L)|}{|B(r)|} \geq \frac{1}{L^N} \liminf_{r\to 0^+} \frac{ |E(\omega)\cap B(x,r)|}{|B(r)|}>0\;,
$$
and, hence, $y$ is also a point of positive density for $E(\omega)$, i.e., (\ref{step2}) holds. \\

{\bf Step 3: } The set $\hat E(\omega)$ coincides a.e. in $\R^N$ with its topological closure $\overline{\hat E(\omega)}$, which is invariant under that action of  $X^\alpha_t$ for any $t>0$ and $\alpha\in {\mathcal A}_0$, i.e., for all $t>0$ and $\alpha\in {\mathcal A}_0$,
\be\label{step3}
{\rm if } \ \  x\in \overline{\hat E(\omega)}, \quad {\rm then} \quad X^{\alpha}_t(x)\in \overline{\hat E(\omega)}. 
\ee

Following the arguments of the proof of Lemma~\ref{lem:cone}, there exists a sufficiently small $\eta>0$ such that, for all $x\in \R^N$, the cone
${\mathcal C}_\eta(x)=x+ (0,\eta)[V(x,\omega)+ B(\bar a, \eta)]$ is contained in $\{ X^\alpha_s(x)\; : \; \alpha\in {\mathcal A}_0, \; s\in (0,h)\}$.
If, furthermore,  $x\in \hat E(\omega)$, then Step 2 implies that the set in the right-hand  side of the above inclusion is also contained in $\hat E(\omega)$. Hence, for all $x\in B(0,R)\cap \hat E(\omega)$,
$$
{\mathcal C}_\eta(x) \subset \hat E(\omega). 
$$
This regularity property easily implies that $|\overline{\hat E(\omega)}\backslash \hat E(\omega)|=0$. Indeed, otherwise, there would exist a Lebesgue point $x$ for the set  $\overline{\hat E(\omega)}\backslash \hat E(\omega)$, i.e., for some $x$ we would have
$$\lim_{r\to 0^+} \frac{|(\overline{\hat E(\omega)}\backslash \hat E(\omega))\cap B(x,r)|}{|B(r)|}= 1\;.$$

Let  $(x_n)_{n\in\N}$ be a sequence in $\hat E(\omega)$ converging to $x$. Then, since ${\mathcal C}_\eta(x_n)\subset  \hat E(\omega)$, we must have
$$
|(\overline{\hat E(\omega)}\backslash \hat E(\omega))\cap B(x,r)|\leq
| B(x,r)\backslash {\mathcal C}_\eta(x_n)| \;.
$$

Letting $n\to\infty$ in the above inequality yields
$$
|(\overline{\hat E(\omega)}\backslash \hat E(\omega))\cap B(x,r)|\leq | B(x,r)\backslash \overline{ {\mathcal C}_\eta(x)} |\;,
$$
and, hence,
$$
1= \lim_{r\to 0^+} \frac{|(\overline{\hat E(\omega)}\backslash \hat E(\omega))\cap B(x,r)|}{|B(r)|}\leq
\limsup_{r\to 0^+} \frac{| B(x,r)\backslash \overline{ {\mathcal C}_\eta(x)}|}{|B(r)|} <1\;,
$$
which is a contradiction. Therefore we must have $|\overline{\hat E(\omega)}\backslash \hat E(\omega)|=0$.

The fact that (\ref{step3}) holds is a straightforward consequence of the same property for $\hat E(\omega)$ proved in Step 2. \\

{\bf Step 4 : } We now complete the proof of the Lemma. It is easily checked that, if $u$ is the indicator function of $\R^N\backslash\overline{\hat E(\omega)}$, then $u$ is a stationary viscosity supersolution to
$$
0 \geq \max_i \{ \lg Du, V+a_i\rg \}\;,
$$
where
$$
\max_i \{ \lg Du, V+a_i\rg \}= \ep  |Du|_\infty + \lg V+\bar a , Du\rg \geq  (\ep/\sqrt{2}) |Du| + \lg V+\bar a , Du\rg\;.
$$

Hence $u$ satisfies, in the viscosity sense,
$$
0\geq |Du|+ \lg \frac{V+\bar a}{(\ep/\sqrt{2}} , Du\rg
$$
where  $(\ep/\sqrt{2})^{-1}{V+\bar a}$  is a stationary, divergence free vector field. In view of Lemma \ref{SubSolOmega}, this implies that $u$ is constant.

So $E(\omega)= \overline{\hat E(\omega)}=\R^N$ a.e. in $\R^N$, and, hence, $\P(E)=1$. \QED

\section{Enhancement of the velocity}\label{sec:enhancement}

In the previous section we proved that the averaged front is governed by a Hamilton-Jacobi equation with a Hamiltonian $\overline H$ satisfying, for all $p\in\R^N$, $\overline H(p)\geq |p|+\lg \E[V],p\rg$.
Here we prove Theorem \ref{theo:hen1} which yields that this inequality is actually strict in all directions ``seen'' by the vector field $V$.
In other words, the presence of the vector field $V$ enhances the speed.

In view of the reduction argument given in the proof of Theorem \ref{theo:hom1}, we may assume (\ref{eq:mz}) throughout this section. 

Before we enter in the proof of Theorem~\ref{theo:hen1}, we present a formal argument that yields the enhancement and  motivates the several technical steps of the rigorous proof.

To this end, for simplicity we take $p=e_1$, assume that $\overline H(e_1)=1$, i.e., that there is no enhancement in the $e_1$ direction, and show (formally) that $V_1=0$.

Assume next that the ``cell problem''
\be\label{eq:cell}
|Dw + e_1| + \lg V(y,\omega), Dw + e_1 \rg = 1  \ \ \text{ in \ $\R^N$  \ and \ a.s. in \ $\omega$},
\ee
has an a.s. Lipschitz continuous solution $w$ with $Dw$ stationary and of mean $0$ and, in addition,  that \eqref{eq:cell} holds a.s. in $\omega$ at $y=0$. Note that the existence of such $w$ is a big assumption. Indeed most probably the claim does not hold even in the periodic setting.


Averaging over $\omega$ and using the properties of $V$ yields
\be\label{eq:cor1}
\E[|Dw + e_1|]\leq 1.
\ee

Employing the elementary fact that
\be\label{eq:in}
(a^2 +b^2)^{1/2} \geq (1-\ep)^{1/2} a + \ep^{1/2} b \ \ \text{ for all $a,b \in \R_+$ and $\ep>0$}
\ee
as well as Jensen's inequality in \eqref{eq:cor1} yields that
$$
1\geq (1-\ep)^{1/2} \E[|\partial_1w+ 1|] + \ep^{1/2} \E[|\hat D  w|].
$$
where $\hat D w=(\partial_{x_2}w, \dots, \partial_{x_N}w)$.
Hence,
$$
\E[|\hat D  w|] =0 \ \ \text{ and } \ \ \E[(\partial_1w+ 1)_-]=0,
$$
and, therefore,
\be\label{eq:cor2}
\hat D  w =0 \ \ \text{ and } \ \ \partial_1w + 1 \geq0 \ \ {a.s. \ in} \ \omega.
\ee

In view of \eqref{eq:cor2}, it follows from \eqref{eq:cell} that
$$
(1+V_1)(\partial_1w +1) = 1 \ \ \text{ a.s. in \ $\omega$},
$$
and, hence,
\be\label{eq:cor3}
V_1 + 1\geq 0 \ \ \text{ a.s. in \ $\omega$}.
\ee

Moreover,
$$
\partial_1 w = - \frac {V_1}{1 +V_1} \ \  \text{ a.s. in \ $\omega$},
$$
and
$$
\E[\partial_1 w]=-\E\left[\frac {V_1}{1 +V_1}\right]=0.
$$

Using again Jensen's inequality (more details are given in the course of the rigorous proof)
leads to $V_1=0$.

The rest of the section is devoted to the proof of Theorem~\ref{theo:hen1}. Since we do not know that correctors, i.e., solutions of \eqref{eq:cell}
with the properties listed above, exist for all $p$, we need to work with the solutions $v_\delta^p$ of
\be\label{eq:deltavdelta}
\delta v^p_\delta = \left| Dv^p_\delta+p\right| + \lg Dv^p_\delta+p, V\rg \ \  {\rm in } \ \  \R^N,
\ee
which, for $\delta>0$ and $p\in\R^N$, (see, for example, \cite{cil}), are a.s. in $\omega$, unique, bounded, continuous and stationary.

As explained in the introduction,  in the periodic setting it is possible to obtain, using the isoperimetric inequality, uniform in $\delta$, bounds for the oscillation of $v^p_\delta$. This in turn implies that the family $(\delta v^p_\delta)_{\delta>0}$ converges, as $\delta\to0$, uniformly in $\R^N$ to a constant and, hence, homogenization takes place. The convergence of the  $\delta v^p_\delta$'s was then used in \cite{cns} to show the enhancement.

In our setting, such estimates were not available and we followed a different approach to prove that the $G$-equation homogenizes. Having
established this fact, we may now go back and ascertain that the $\delta v^p_\delta$'s converge, as $\delta\to0$, in the appropriate sense to $\overline H(p)$. The exact statement is in the next lemma, which we state without proof. For the latter we refer to \cite{LS} and \cite{AS}.

We have:

\begin{lem} For each $p\in \R^N$ and $R>0$, as $\delta\to0$, $\delta v^p_\delta \to \overline H(p)$ uniformly in balls $B(y/\delta,R/\delta)$ for $y$ bounded and a.s. in $\omega$.
\end{lem}

We continue with the

\noindent {\bf Proof of Theorem~\ref{theo:hen1}: }

First assume that $\lg V(\cdot,\omega), p\rg=0$ in $\R^N$ and a.s. in $\omega$. Then $v^p_\delta(x,\omega)= \frac{|p|}{\delta}$ is the unique solution to (\ref{eq:deltavdelta}), and, hence, $\bar H(p)= 1$.

For the rest of  the proof we assume, for simplicity,  that $p=e_1$ and we denote by $v_\delta$ the solution of (\ref{eq:deltavdelta}) with $p=e_1$,
i.e., $v_\delta=v_\delta^{e_1}$. The result will come from the analysis of the behavior of the $\delta v_\delta$'s as $\delta\to 0$. In order to avoid measurability issues, we will always work up to subsequences as $\delta\to 0$ (although we will not write it explicitly for notations sake).

Next we assume that  $\overline H(e_1)=1$ and we show that $V_1=0$ a.s. in $\omega$.
The proof follows the formal arguments presented earlier. But making these rigorous is rather involved. For the arguments below it is necessary to use  that (\ref{eq:deltavdelta}) holds in the a.e. sense (viscosity is not enough),  and to let $\delta\to0$. Due to the lack of estimates we need to regularize the $v_\delta$'s and the sup-convolution (see \eqref{eq:conv}) is the right way to do it. To this end,  for $\eta>0$ small, let
$$
\tilde v_{\delta,\eta}(x,\omega)= \sup_{y\in \R^N} \left\{ v_\delta(y,\omega)- (2\eta)^{-1}|x-y|^2\right\}\;.
$$
It follows that the $\tilde v_{\delta,\eta}$'s are Lipschitz continuous, stationary and, as $\eta\to 0$, $\tilde v_{\delta,\eta}\to v_\delta$ and
$\delta\tilde v_{\delta,\eta}\to \delta v_\delta$ locally uniformly in $\R^n$ and a.s. in $\omega$. Notice that the  $\delta \tilde v_{\delta,\eta}$'s are uniformly bounded uniformly in $\eta$. Therefore, the convergence as $\eta\to 0$, of the $\delta \tilde v_{\delta,\eta}$'s to $\delta v_\delta$ also holds in $L^1(\Omega)$. Hence, we can choose $\eta>0$ sufficiently small so that, if $\|f\|_1$ denotes the $L^1(\Omega)$-norm of $f:\Omega\to\R$,
$$\|\delta \tilde v_{\delta,\eta}-\delta v_\delta\|_1\leq \|\delta v_\delta-1\|_1\;.$$
Next we use the fact that the sup of solutions of Hamilton-Jacobi equations  with concave Hamiltonians is still a solution (of a modified equation) (see Barron-Jensen \cite{bj90}, Barles \cite{Ba93}). Using the argument in the proof of Lemma~\ref{SubSolOmega} we can find $\gamma_{\delta,\eta}>0$ such that $\tilde v_{\delta,\eta}$ satisfies, a.e. in $\R^N$ and a.s. in $\omega$,
$$
\delta \tilde v_{\delta,\eta} \geq  (1-\gamma_{\delta,\eta} ) \left| D\tilde v_{\delta,\eta} +e_1\right| + \lg D\tilde v_{\delta,\eta} +e_1, V\rg,
$$
and
$$
\delta \tilde v_{\delta,\eta} \leq  (1+\gamma_{\delta,\eta} ) \left| D\tilde v_{\delta,\eta} +e_1\right| + \lg D\tilde v_{\delta,\eta} +e_1, V\rg .
$$
Note also that, since $\gamma_{\delta,\eta}\to 0$ as $\eta\to 0$, we can choose $\eta=\eta(\delta)$ such that, as $\delta\to 0$, $\gamma_{\delta,\eta(\delta)}\to0$. 

From now on, in order to simplify the presentation, we omit the dependence of the various quantities on $\eta$ and we set $\tilde v_\delta=
\tilde v_{\delta,\eta(\delta)}$ and $\gamma_\delta= \gamma_{\delta,\eta(\delta)}$.

For later use  we rewrite the previous observations with the new notation, i.e., as $\delta\to0$,  $\gamma_\delta\to 0$, while $\tilde v_\delta$ satisfies
a.e. in $\R^N$ and a.s. in $\omega$,
\be\label{ineq1}
\delta \tilde v_\delta \geq  (1-\gamma_{\delta} ) \left| D\tilde v_{\delta} +e_1\right| + \lg D\tilde v_{\delta} +e_1, V\rg,
\ee
and
\be\label{ineq2}
\delta \tilde v_\delta \leq  (1+\gamma_{\delta} ) \left| D\tilde v_{\delta} +e_1\right| + \lg D\tilde v_{\delta} +e_1, V\rg.
\ee

A simple Fubini Theorem-type argument also yields that we may assume that  \eqref{ineq1} and \eqref{ineq2} hold a.s. in $\omega$ for some (independent of $\omega$) $x_0$. For notational simplicity below we take $x_0=0$ and assume, although we do not write it explicitly, that  \eqref{ineq1} and \eqref{ineq2}
hold for $x=0$ and a.s. in $\omega$. Finally recall that, since $\tilde v$ and its gradient are both stationary functions, any integral norm in $\omega$ is independent of where the function is evaluated in space. Hence below, when we write $L^1$-norms, we omit this dependence.


The proof is based on the three lemmata which we state next. We present their proofs after the end of the ongoing one.

We have:

\begin{lem}\label{lem:1} As $\delta \to 0$,  $\|\hat D \tilde v_\delta\|_1\to 0$ and $\|\left(\partial_{x_1} \tilde v_\delta+1\right)_-\|_1\to 0$. 
\end{lem}

\begin{lem}\label{lem:dx1v+1} There exists a set of full measure $\Omega_{0}\subseteq\Omega$ such that, for all $\omega\in \Omega_{0}$,\\  $\partial_{x_1}\tilde v_\delta(0,\omega)+1\geq 0$ for $\delta>0$ sufficiently small.
\end{lem}

\begin{lem}\label{V1geq-1} For all $\omega\in\Omega_{0}$ given in the previous lemma, $V_1(0,\omega)\geq -1$
\end{lem}

Let
$$
E_\delta =\{\omega\in \Omega_{0}\; : \; \partial_{x_1}\tilde v_\delta(0,\omega)+1\geq 0\}\;,
$$
and note that Lemma \ref{lem:dx1v+1} yields that, as $\delta\to 0$, $\P[E_\delta]\to 1$.

Since  (\ref{ineq2}) implies that, for any $\sigma>0$ and a.s. in $\omega$,
$$
\delta \tilde v_\delta \leq  (1+\gamma_{\delta}+\sigma ) \left|\partial_{x_1} \tilde v_{\delta} +1\right|+ (1+\gamma_{\delta})\left| \hat D\tilde v_{\delta} \right| + \lg D\tilde v_{\delta} +e_1, V\rg\;,
$$
on $E_\delta$ we have
$$
\delta \tilde v_\delta \leq  (1+\gamma_{\delta} +\sigma) \left( \partial_{x_1} \tilde v_{\delta} +1\right)
+ |\hat D \tilde v_\delta| (1+\gamma_\delta+ \|V\|_\infty) + \left( \partial_{x_1}\tilde v_{\delta} +1\right) V_1.
$$

In view of Lemma \ref{V1geq-1} we know that $1+V_1\geq 0$. Hence, on $E_\delta$,
$$
(1+\gamma_{\delta} +\sigma+V_1)^{-1}(\tilde v_\delta -|\hat D \tilde v_\delta| (1+\gamma_\delta+ \|V\|_\infty)-(1+\gamma_{\delta} +\sigma)-V_1)
\leq \partial_{x_1} \tilde v_{\delta}.
$$

Integrating over $\omega$ this last  inequality we get
$$
\E[(1+\gamma_{\delta} +\sigma+V_1)^{-1}(\tilde v_\delta -|\hat D \tilde v_\delta| (1+\gamma_\delta+ \|V\|_\infty)-(1+\gamma_{\delta} +\sigma)-V_1){\bf 1}_{E_\delta}]\leq \E\left[ \partial_{x_1} \tilde v_{\delta}{\bf 1}_{E_\delta}\right]\;.
$$

Note that, since
$$
0=\E[ \partial_{x_1} \tilde v_{\delta}]= \E[ \partial_{x_1} \tilde v_{\delta}{\bf 1}_{E_\delta}]+ \E[(\partial_{x_1} \tilde v_{\delta}+1){\bf 1}_{E_\delta^c}]
-\P[E_\delta],
$$
and, as $\delta\to0$,  
$$
\P[E_\delta]\to 0 \ \text{ and } \ \E[(\partial_{x_1} \tilde v_{\delta}+1){\bf 1}_{E_\delta^c}]=
-\E[(\partial_{x_1} \tilde v_{\delta}+1)_-]\to 0, 
$$
always as $\delta\to0$, we have
 $$\E[ \partial_{x_1} \tilde v_{\delta}{\bf 1}_{E_\delta}]\to0.$$

Observe also that, since $1+V_1+\sigma \geq \sigma>0$, as $\delta\to0$,
$$
\E[-(1+\gamma_{\delta} +\sigma+V_1)^{-1}(\gamma_{\delta} +\sigma +V_1){\bf 1}_{E_\delta}] \to \E[-(1+\sigma+V_1)^{-1}(\sigma +V_1)]
$$
while $\|\hat D \tilde v_\delta\|_{L^1(\Omega)}\to 0$ and $\|\delta \tilde v_\delta\|_{L^1(\Omega)}\to 1$, as $\delta\to0$, imply
$$
\E\left[(1+\gamma_{\delta} +\sigma+V_1)^{-1}(\delta \tilde v_\delta-|\hat D \tilde v_\delta| (1+\gamma_\delta+ \|V\|_\infty)-1)
{\bf 1}_{E_\delta}\right] \to 0 .
 $$

Combining all the above we find that
$$
\E\left[(1+\sigma+V_1)^{-1}(\sigma + V_1) \right]\geq 0\;.
$$

Next for $s>-1$, let $\Phi(s)= -\frac{s}{1+s}=-1+ \frac{1}{1+s}$.  The monotone convergence theorem implies that, as $\sigma\to 0$,
$$
\E\left[\Phi(V_1)\right]=\E\left[\frac{-V_1}{1+V_1} \right] \leq 0\;.
$$

Since $\Phi$ is convex, using Jensen's inequality and the facts that $\E[V_1]=0$ and $\Phi(0)=0$, we get
$$
0=  \Phi( \E[V_1]) \leq \E[\Phi( V_1)]\leq 0\;.
$$

But $\Phi$ is actually strictly convex, so that  the equality  $\Phi( \E[V_1])=  \E[\Phi( V_1)]$ implies that $V_1$ must be constant, and, therefore,  $0$.
\QED

We return to the proofs of the three lemmata used in the course of the proof of Theorem~\ref{theo:hen1} and we begin with the

\noindent {\bf Proof of Lemma~\ref{lem:1}: }
Using in \eqref{eq:in} 
we find that, a.s. in $\omega$ and a.e. in $\R^N$,
$$
\delta \tilde v_\delta \geq  (1-\gamma_{\delta} )(1-\ep)^{\frac12} \left| \partial_{x_1} \tilde v_{\delta} +1\right| +
 (1-\gamma_{\delta} )\ep^{\frac12} \left|\hat  D\tilde v_{\delta} \right|
+ \lg D\tilde v_{\delta} +e_1, V\rg . 
$$

Since $V$ is divergence free and has mean zero, averaging over balls $B_R$ and letting $R\to\infty$, we obtain, using the ergodic theorem,
$$
\E[ \lg D\tilde v_{\delta} +e_1, V\rg]= -\E[\tilde v_\delta {\rm div}(V)]+\lg e_1, \E[V]\rg = 0\;.
$$

Recalling that, as $\delta\to$, $\delta \tilde v_\delta\to 1$ in $L^1(\Omega)$ and, in addition, $\E\left[\partial_{x_1} \tilde v_{\delta}\right] =0$,
we get
$$
\begin{array}{rl}
1+ o(1) \;  \geq  & \ds  (1-\gamma_{\delta} )(1-\ep)^{\frac12} \E[| \partial_{x_1} \tilde v_{\delta} +1|]+
 (1-\gamma_{\delta} )\ep^{\frac12} \E[|\hat  D\tilde v_{\delta} ] \\
 & \geq  \ds  (1-\gamma_{\delta} )(1-\ep)^{\frac12} | \E[\partial_{x_1} \tilde v_{\delta}] +1| +
 (1-\gamma_{\delta} )\ep^{\frac12} \E[|\hat  D\tilde v_{\delta}|] \\
 & \geq  \ds  (1-\gamma_{\delta} )(1-\ep)^{\frac12} +
 (1-\gamma_{\delta} )\ep^{\frac12} \E[|\hat  D\tilde v_{\delta}|].
 \end{array}
$$

Choosing $\ep_\delta>0$ such that, as $\delta\to0$,
$$
((1-\gamma_{\delta} )\ep_\delta^{\frac12})^{-1}(1+o(1)-(1-\gamma_{\delta} )(1-\ep_\delta)^{\frac12}) \to 0,
$$
for example, let  $\ep_\delta = |\theta_\delta|^{\frac12}+\gamma_\delta^{\frac12}$ with $\theta_\delta\to0$,  we get, that, as $\delta\to0$,
$\E[|\hat  D\tilde v_{\delta}|]\to 0.$

Using in (\ref{ineq1}) that $s\to s_+$ is convex and $ \E\left[\partial_{x_1} \tilde v_{\delta} \right] =0$ we obtain, for some $o(1)\to0$ as $\delta\to0$,
$$
\begin{array}{rl}
1+ o(1) \;  \geq  & \ds  (1-\gamma_{\delta} )\E[| \partial_{x_1} \tilde v_{\delta} +1|]\\
\geq & (1-\gamma_{\delta} )(\E[(\partial_{x_1} \tilde v_{\delta} +1)_-]
+\E[( \partial_{x_1} \tilde v_{\delta} +1)_-])\\
\geq  & (1-\gamma_{\delta} )( 1+\E[( \partial_{x_1} \tilde v_{\delta} +1)_-]).
\end{array}
$$

Hence, as $\delta\to0$,
$$
\E\left[\left( \partial_{x_1} \tilde v_{\delta} +1\right)_- \right] \leq (1-\gamma_{\delta})^{-1}(\gamma_\delta+ o(1))\to 0 .
$$
\QED

We continue with the

\noindent {\bf Proof of Lemma~\ref{lem:dx1v+1}: }
In view of Lemma~\ref{lem:1} there exits $\Omega_{0}\subseteq\Omega$ of full measure and a subsequence, which for notational simplicity we still denote with $\delta$, such that, a.s. in $\Omega_0$, $\hat D \tilde v_\delta(0,\cdot)\to 0$ and $\left(\partial_{x_1} \tilde v_\delta(0,\cdot)+1\right)_-\to 0$ in $\Omega_{0}$.

Fix $\omega\in \Omega_{0}$. If, up to a subsequence, $\partial_{x_1}\tilde v_\delta(0,\omega)+1<0$, then
  $\partial_{x_1}\tilde v_\delta(0,\omega)+1\to 0$. Since $\hat D \tilde v_\delta(0,\omega)\to 0$, letting $\delta\to 0$ in (\ref{ineq2}) implies
$ 1\leq 0\;,$ an obvious contradiction.
\QED

We conclude with the

\noindent {\bf Proof of Lemma~\ref{V1geq-1}: }
Fix $\omega\in \Omega_{0}$. Lemma~\ref{lem:dx1v+1} implies that, for $\delta>0$ sufficiently small,  $\partial_{x_1}\tilde v_\delta(0,\omega)+1\geq 0$.
Then, (\ref{ineq2}) evaluated at $(0,\omega)$ gives
$$
\delta \tilde v_\delta \leq  (1+\gamma_{\delta} ) \left( \partial_{x_1} \tilde v_{\delta} +1\right)
+ |\hat D \tilde v_\delta| (1+\gamma_\delta+ \|V\|_\infty) + \left( \partial_{x_1}\tilde v_{\delta} +1\right) V_1
$$

Assume that $V_1(0,\omega)+1<0$. Then,  for $\delta$ small enough, we find
$$
(1+\gamma_{\delta}+V_1(0,\omega))^{-1}(\delta \tilde v_\delta(0,\omega) - |\hat D \tilde v_\delta(0,\omega)| (1+\gamma_\delta+ \|V\|_\infty) )\geq \partial_{x_1} \tilde v_{\delta}(0,\omega) +1\;.
$$

The left-hand side in the above inequality converges, as $\delta\to0$, to $(1+V_1(0,\omega))^{-1} <0$,  while the right-hand side is nonnegative, a contradiction.
\QED

\section{Appendix: A simple proof of the homogenization when $N=2$}\label{sec:N=2}

We present here a shorter and simpler argument of the homogenization result when $N=2$. Here again it will be convenient to work under the additional assumption that $V$ has mean zero, which, of course, can be removed  as in Section \ref{sec:genN}. The simpler proof relies on two facts. The first, which holds in any dimension, is that, for each $a\in B$, $r^{-1}\theta(0, rp,\omega)$ has a.s a limit, as $r\to+\infty$, for any direction $p$ belonging to the essential support of
the random variable $Z^a=\E[ V+a\;|\; {\mathcal F}^a]$, where ${\mathcal F}^a$ is the invariant sets for the measure preserving transformations $(T^a_{t})$ (see the proof of Lemma \ref{lem:B01subsetE-}).
This is especially useful when $N=2$. Indeed the second fact is that, on the plane and for any $a\in B$, the support of $Z^a$ is contained in $\R a$. Hence, since $\E[Z^a]=a$, the support must contain $\lambda a$ for some $\lambda\geq 1$. Combining these two facts one can show that the limit $r^{-1}\theta(0, rp,\omega)$  exists for any $p$. \\

As we already pointed out in the proof of Lemma \ref{lem:B01subsetE-}, for any fixed vector $a\in B$ and a.s. in $\omega$,
$\lim_{t\to\infty}t^{-1}X_t^{0,0,a,\omega}=Z^a=\E[ V+a\;|\; {\mathcal F}^a]$.  

Let $W^a$ denote the essential support of $Z^a$, i.e.,
$$
W^a=\left\{ w\in \R^N : \P[\{\omega: |Z^a(\omega)-w|<\ep\}]>0 \ \text{ for all $\ep>0$} \right\}\;.
$$
Note that, since $\E[Z^a]=\E[\E[V+a|{\mathcal F}^a]]=a$, $W^a$ must contain some $w$ such that $\lg w, a\rg \geq 1$.

We have:

\begin{lem} \label{CvForWalpha}
For any $w\in W^a$,
$
\overline q(w)= \lim_{t\to+\infty} t^{-1}\theta(0,tw,\omega)
$
exists a.s. in $\omega$ and in all dimensions.
\end{lem}

\noindent {\bf Proof : }
If
$
T^a(s,t,\omega)=\theta(X^{0,0,a,\omega}_s, X^{0,0,a,\omega}_t,\omega),
$
we claim that there exists $\Gamma^a:\Omega\to\R$ such that, a.s. in $\omega$ and as $t\to\infty$,
\be\label{1/tT}
t^{-1} T^a(0,t,\omega) \to \Gamma^a(\omega). 
\ee

To prove (\ref{1/tT}) we first note that, for any $\sigma>0$, if $x_\sigma=x_\sigma(\omega)=X^{0,0,a,\omega}_\sigma$, then
$$
T^a(s+\sigma,t+\sigma,\omega) =
\theta (X^{x_\sigma,0,a,\omega}_s, X^{x_\sigma,0,a,\omega}_t,\omega).
$$

We also know that
$$
\theta (X^{x_\sigma,0,a,\omega}_s, X^{x_\sigma,0,a,\omega}_t,\omega)
=\inf\left\{r\geq 0\;:\; X^{x_\sigma,0,a,\omega}_t \in {\mathcal R}_r(X^{x_\sigma,0,a,\omega}_s, \omega)\right\},
$$
with
$$
X^{x_\sigma,0,a,\omega}_t= X^{0,0,a,\tau_{x_\sigma}\omega}_t+x_\sigma(\omega),
$$
while
$$
{\mathcal R}_r(X^{x_\sigma,0,a,\omega}_s, \omega) =
{\mathcal R}_r(X^{0,0,a,\tau_{x_\sigma}\omega}_s, \tau_{x_\sigma}\omega)+x_\sigma(\omega).
$$

Hence,
$$
\theta (X^{x_\sigma,0,a,\omega}_s, X^{x_\sigma,0,a,\omega}_t,\omega)
=
\theta (X^{0,0,a,\tau_{x_\sigma}\omega}_s, X^{0,0,a,\tau_{x_\sigma}\omega}_t,\tau_{x_\sigma}\omega).
$$

Since the map $\omega\to \tau_{x_\sigma(\omega)}\omega$ is measure preserving, the process $T^a$ is stationary. It is also clearly sub-additive and, in addition, for $0<s<t$, we have
$$
T^a(s,t,\omega)\leq t-s\;.
$$

Then \eqref{1/tT} follows immediately from the
sub-additive ergodic theorem.


Fix  $w\in W^a$ and $\ep>0$. The  definition of $W^a$ yields that the set
$S_\ep=\left\{\omega\in \Omega\;:\;  |Z^a(\omega)-w|<\ep\right\}$ has a positive measure. Fix $\omega\in S_\ep$
such that, $\lim_{t\to\infty}{t}^{-1} X^{0,0,a,\omega}_t\to Z^a(\omega)$ and  let $T(\omega,\ep)\in(0,\infty)>0$  be given by
Theorem \ref{reach}.
Then
$$
\begin{array}{rl}
\theta(0, tw,\omega)\leq & \theta (0, X^{0,0,a,\omega}_t,\omega)+
\theta (X^{0,0,a,\omega}_t,tw,\omega)\\[2mm]
\leq &  T^a(0,t, \omega)
+T(\omega)+\ep|X^{0,0,a,\omega}_t|+ (1+\ep)|X^{0,0,a,\omega}_t-tw|,
\end{array}
$$
and, thus,
$$
\limsup_{t\to+\infty}\frac{1}{t}\theta(0, tw,\omega)
\leq \limsup_{t\to+\infty}\frac{1}{t}  T^a(0,t,\omega)+
M\ep+ (1+\ep)\ep
\leq Z^a(\omega)+(1+\ep)\ep\;. $$

Similarly, we find
$$
\liminf_{t\to+\infty}\frac{1}{t}\theta(0, tw,\omega)
\geq  Z^a(\omega)-(1+\ep)\ep\;.
$$

Hence,
 $$
 \limsup_{t\to+\infty}\frac{1}{t}  \theta(0, tw,\omega) -\liminf_{t\to+\infty}\frac{1}{t}  \theta(0, tw,\omega)  \leq 2(1+\ep)\ep\;,
 $$
which shows the equality between $\limsup$ and $\liminf$ since $\ep$ is arbitrary.
 \QED

When $N=2$, $W^a$ consists only of a  direction parallel to $a$. Indeed we have:

\begin{lem}\label{2D} If $N=2$, then $W^a\subset \R a$ and there exists $\lambda \geq 1$ such that
$\lambda a\in W^a$.
\end{lem}

We remark that the proof actually shows that there exists a random variable $\lambda$ such that
$$
\lim_{t\to \pm \infty} \frac{1}{t} X^{0,0,a,\omega}_t = \lambda(\omega) a \qquad {\rm a.s.}\;
\qquad {\rm with }\; \ \P[\lambda\geq 1]>0\;.
$$

Before we present the proof of this last lemma, we observe that
combining Lemma \ref{CvForWalpha} and Lemma \ref{2D} easily provides the existence of a time constant, exactly as in the proof of Theorem \ref{thm:TimeConstant}.

We have:

\begin{cor} If $N=2$, for any $a\in B$, $\lim_{t\to+\infty}t^{-1} \theta(0,ta,\omega)$ exists a.s. in $\omega$.
\end{cor}

We conclude with the

\noindent {\bf Proof of Lemma \ref{2D} : } First we prove that $W^a$ is contained in a single line. If not there must exist
$\bar z_1,\bar z_2\in W^a$ such that $\bar z_1\neq 0$, $\bar z_2\neq 0$ and, if
we set $\bar z_1^\perp:= J\bar z_1$, where $J=\left( \begin{array}{ll} -1 & 0\\0 & 1\end{array}\right)$,  then
$\lg \bar z_1^\perp, \bar z_2\rg\neq 0$. We choose
$\ep>0$ so small that $\lg z_1^\perp, z_2\rg\neq 0$ for any $z_1\in B(\bar z_1,\ep)$ and $z_2\in B(\bar z_2,\ep)$
and consider the set
$$
E_2= \{ \omega \in \Omega\;:\; | \lim_{t\to\pm\infty} \frac{1}{t} X^{0,0,a,\omega}_t-\bar z_2| \leq \ep \}\;.
$$

The definition of $W^a$ implies that $\P(E_2)>0$. As usual we set
$E_2(\omega)= \left\{ x\in \R^2 : \tau_x\omega\in E_2 \right\}$ and fix $\omega$ such that
$$
z_1= \lim_{t\to\pm\infty} \frac{1}{t} X_t^{0,0,a,\omega}
$$
exists and belongs to $B(\bar z_1,\ep)$ (hence $z_1$ is non zero) and, in addition,
$$
\lim_{R\to+\infty} \frac{ |Q_R\cap E_2(\omega)|}{|Q_R|}= \P\left[ E_2\right]>0\;.
$$

The curve
$L=\{X_t^{0,0,a,\omega} : t\in \R\}$ (recall $z_1\neq 0$) divides the plane into two connected components $\Pi^+$ and $\Pi^-$. Moreover, for any $\sigma>0$
small, there exists some $\delta>0$
such that
\be\label{Pi+Pi-}
\Pi^+\subset \left\{x\in \R^2\;:\; \lg x, z_1^\perp\rg \geq -\delta-\sigma|x|\right\}
\;{\rm and } \;
\Pi^-\subset \left\{x\in \R^2\;:\; \lg x, z_1^\perp\rg \leq \delta+\sigma|x|\right\}\;.
\ee

Choose $R>0$ sufficiently large so that $Q_R\cap E_2(\omega)\neq \emptyset$. If $x_2\in E_2(\omega)$, then, for some $z_2\in B(\bar z_2,\ep)$,
$$
\lim_{t\to\pm\infty} \frac{1}{t} X_t^{x_2,0,a,\omega}
=
\lim_{t\to\pm\infty} \frac{1}{t} \left( X_t^{0,0,a,\tau_{x_2}\omega}+x_2\right)
= z_2.
$$

Since $z_2\neq z_1$, the trajectory $X_t^{x_2,0,a,\omega}$ never crosses
the curve $L$ and therefore remains either in $\Pi^+$ or in $\Pi^-$. But then the fact that
$$
\lim_{t\to\pm\infty} \frac{1}{t} \lg X_t^{x_2,0,a,\omega}, z_1^\perp\rg
= \lg z_2, z_1^\perp\rg \neq 0
$$
contradicts (\ref{Pi+Pi-}) as soon as $\sigma$ is sufficiently small.

So we have proved that there is a direction $z\neq 0$ such that $W^a\subset \R z$, i.e., a.s. in $\omega$,
$\lim_{t\to\pm\infty} \frac{1}{t} X_t^{0,0,a,\omega} \in \R z$. Since, from (\ref{ergo+}),
$$
\E\left[ \lim_{t\to\pm\infty} \frac{1}{t} X_t^{0,0,a,\omega}\right] =\E[\E[ V+a\ |\ {\mathcal F}^a]]= a\;,
$$
we can choose $z=a$. This also implies the existence of some $\lambda\geq 1$ such that
$\lambda a\in W^a$.
\QED



\end{document}